\documentclass[11pt]{amsart}
\usepackage{amsmath,amssymb,enumerate}
\usepackage{latexsym}
\usepackage{amsfonts}
\usepackage {amsbsy}
\usepackage {amsmath}
\usepackage{amssymb}
\usepackage{enumerate}
\usepackage{ bbm}

\newtheorem{theorem}{Theorem}
\newtheorem{lemma}[theorem]{Lemma}
\newtheorem{corollary}[theorem]{Corollary}
\newtheorem{proposition}[theorem]{Proposition}

\theoremstyle{definition}
\newtheorem{definition}[theorem]{Definition}
\newtheorem{example}[theorem]{Example}
\newtheorem{remark}[theorem]{Remark}

\numberwithin{theorem}{section}
\numberwithin{equation}{section}

\theoremstyle{plain}
\theoremstyle{definition}

\newcommand{\B}{\mathbb{B}}
\newcommand{\N}{\mathbb{N}}
\newcommand{\R}{\mathbb{R}}
\newcommand{\C}{\mathbb{C}}

\newcommand{\e}{\varepsilon}

\newcommand{\Om}{\Omega}

\makeatletter
\@namedef{subjclassname@2010}{%
  \textup{2010} Mathematics Subject Classification}
\makeatother

\begin{document}

\thanks{The second author was partially supported by the ANR project GRACK}



\title[The continuous subsolution problem]{ The  Continuous Subsolution problem  \\ for Complex Hessian Equations }

\author{Mohamad Charabati }

\address{Al-Bath University, Homs p. 77, Syrian Arab Republic}

\email{charabati.mohamad@gmail.com}

\author{Ahmed Zeriahi}

\address{Institut de Mathématiques de Toulouse; UMR 5219, Université de Toulouse; CNRS, UPS, 118 route de Narbonne, F-31062 Toulouse Cedex 9, France}

\email{ahmed.zeriahi@math.univ-toulouse.fr}

\date{\today}

\begin{abstract}
 Let  $\Omega \subset \mathbb C^n$ be a bounded strictly $m$-pseudoconvex domain ($1\leq m\leq n$) and  $\mu$ a positive Borel measure on $\Omega$. 
 
  We   study the Dirichlet problem for  the complex Hessian equation $(dd^c u)^m \wedge \beta^{n - m} = \mu$ on $\Omega$. 
 
 First we give a sufficient condition on the "modulus of diffusion" of the measure $\mu$ with respect to the $m$-Hessian capacity which guarantees the existence of a continuous solution to the associated Dirichlet problem with a continuous boundary datum.
 
  As an application, we prove that if the equation has a continuous $m$-subharmonic subsolution whose modulus of continuity satisfies a Dini type condition, then the equation has a continuous solution with an arbitrary continuous boundary datum. Moreover when the measure has a finite mass on $\Omega$, we give a precise quantitative estimate on the modulus of continuity of the solution. 
 
  One of the main steps in our proof is to establish a new capacity estimate providing a precise estimate of the modulus of diffusion of the  $m$-Hessian measure of a continuous $m$-subharmonic function $\varphi$ in $\Omega$ with zero boundary with respect to the $m$-Hessian capacity in terms of the modulus of continuity of $\varphi$.
   Another important ingredient is a new weak stability estimate for the $m$-Hessian measure of a continuous $m$-subharmonic function in $\Omega$. 
\end{abstract} 

\maketitle

\tableofcontents

\section*{Introduction}

Complex Hessian equations are important examples of fully non-linear PDE's of second order on complex manifolds.  They interpolate between (linear) complex Poisson equations ($m = 1$) and (non linear) complex Monge-Amp\`ere equations ($m=n$).
 They arise in many geometric problems, including the $J$-flow
\cite{SW} and quaternionic geometry \cite{AV}. They have attracted the attention of many researchers these last years as we will explain below. 

 \subsection{Statement of the problem}
 
Let $\Omega \Subset \C^n$ be a bounded domain and $m$  a fixed integer such that $1 \leq m \leq n$. 
We consider the following general Dirichlet problem for the complex $m$-Hessian equation : 

\smallskip
\smallskip

{\it The Dirichlet problem:} Let $g \in \mathcal{C}^{0} (\partial \Omega)$ be a continuous function (the boundary  datum) and $\mu$ a  positive Borel measure  on $\Omega$ (the right hand side). The Dirichlet problem with boundary datum $g$ and right hand side $\mu$ consists in finding a function $ U \in \mathcal{SH}_m (\Omega) \cap \mathcal C^{0} (\Omega)$  satisfying the following properties :   

\begin{equation}\label{eq:DirPb}
\left\{\begin{array}{lcl} 
 (dd^c U)^m \wedge \beta^{n - m} = \mu,  &\hbox{on}\  \Omega, \, \, \, \, \, (\dag)\\
  U_{\mid  \partial \Omega} = g, & \hbox{on}\  \partial \Omega, \, \, \, (\dag \dag)
\end{array}\right.
\end{equation}
 
The  equation $(\dag)$ must be understood in the sense  of currents  on $\Omega$ (see section 2). 

The equality $ (\dag \dag)$ means that $\lim_{z \to \zeta} U (z) = g (\zeta)$ for any $\zeta \in \partial \Omega$. 

Observe that the comparison principle implies the uniqueness of the solution to the Dirichlet problem  (\ref{eq:DirPb})  when it exists. We will denote it by $U_{g,\mu} = U^{\Omega}_{g,\mu}$.

Recall the usual notations $d = \partial + \bar{\partial}$ and $d^c := (i\slash 2)  ( \bar{\partial} - \partial)$ so  that
$dd^c = i  \partial  \bar{\partial}$.  Given  a real function $u \in \mathcal{C}^2 (\Omega)$,  for each integer $1 \leq k \leq n$, we denote by $\sigma_k (u)$ the continuous function defined 	at each point $z \in  \Omega$ as the $k$-th symmetric polynomial of the eigenvalues $\lambda_1 (z) \leq \cdots \leq \lambda_n(z)$ 
of the complex Hessian matrix $  \left(\frac{\partial^2 u }{\partial z_j \partial \bar{z}_k} (z)\right)$ of $u$ i.e. 
$$
\sigma_k (u) (z) := \sum_{1 \leq j_1 < \cdots < j_k \leq n} \lambda_{j_1} (z) \cdots \lambda_{j_k} (z), \, \,  \, \, z \in \Omega.
$$
A simple computation shows that  
$$
(dd^c u)^k \wedge \beta^{n - k} = \frac{(n-k)! \, k!}{n!}  \, \sigma_k (u)  \,  \beta^n, \, \, 
$$
pointwise on $\Omega$ for $1 \leq k \leq m$, where $\beta := dd^c \vert z\vert^2$ is the usual K\"ahler form on $\C^n$.

We say that a real function $u \in \mathcal{C}^2 (\Omega)$ is $m$-subharmonic in $\Omega$ if for any $1 \leq k \leq m$, we have $\sigma_k (u) \geq 0$ pointwise  in $\Omega$

Observe that  the function $u$ is $1$-subharmonic  on $\Omega$ ($m= 1$) if it is  subharmonic in $\Omega$ and $\sigma_1 (u) = (1 \slash 4) \Delta u$, while $u$ is  $n$-subharmonic  in $\Omega$ ($m = n$) if  $u$ is  plurisubharmonic in $\Omega$ and $\sigma_n (u)  = \mathrm{det}  \left(\frac{\partial^2 u }{\partial z_j \partial \bar{z}_k} (z)\right)$.

It was shown by Z. B\l ocki  in \cite{Bl05}, that it is possible to define a general notion of $m$-subharmonic function using the concept of $m$-positive currents (see section 2). Moreover, identifying positive $(n,n)$-currents with positive Radon measures, it is possible to define the $k$-Hessian measure $(dd^c u)^k \wedge \beta^{n - k}$ when $1 \leq k \leq m$ for any (locally) bounded $m$-subharmonic function $u$ on $\Omega$ (see section 2). 

\smallskip

Several questions related to the Dirichlet problem  (\ref{eq:DirPb}) can be addressed.

1. The first problem is to find a necessary and sufficient condition on $\mu$ which guarantees the existence of a solution to the Dirichlet problem (\ref{eq:DirPb}).

2.  The second problem is to study the regularity of the solution $U^{\Omega}_{g,\mu}$ in terms of the regularity of the data $(g,\mu)$.


\smallskip

When $\Omega$ is a smooth strictly $m$-pseudoconvex domain, the boundary data is smooth and the measure $\mu$ has a smooth positive density, S. Li proved in \cite{Li04} that the Dirichlet problem has a unique smooth solution.

Later Z. Blocki  introduced in \cite{Bl05}  the notion of weak solution for the Hessian equation and solved the Dirichlet problem for the homogenuous complex Hessian equation.
When $\mu = f \lambda_{2n}$ has a continuous density $f \in C^0(\bar \Omega)$, M. Charabati solved  the Dirichlet problem (\ref{eq:DirPb}) in \cite{Ch16a}  using the Perron method and gave a precise control of the modulus of continuity of the solution in terms of the moduli of continuity of the data $(f,g)$.

When $\mu = f \lambda_{2n}$ has a density  $ f \in L^p (\Omega)$ with $p > n\slash m$, S. Dinew and S. Ko\l odziej proved in \cite{DK14} that  the Dirichlet problem (\ref{eq:DirPb}) has a continuous solution. Assuming $g$ is Hölder continuous, N.C. Nguyen \cite{N14} proved that the solution is Hölder continuous under and additional condition on the growth of $f$ at the boundary and M. Charabati \cite{Ch16b} proved the Hölder continuity in the general case $ f \in L^p (\Omega)$ with $p > n\slash m$.

\smallskip

When $\mu$ is a {\bf more general} positive Borel measure on $\Omega$ and $g$ is a continuous boundary datum, the Dirichlet problem is much more difficult. A necessary condition for the existence of a solution to (\ref{eq:DirPb}) is  the existence of a subsolution.  

For the complex Monge-Amp\`ere equation, S. Ko\l odziej proved in \cite{Kol95} that if the Dirichlet problem (\ref{eq:DirPb}) has a bounded subsolution, then it has a bounded solution. The same result for the Hessian equation was proved by  N. C. Nguyen in \cite{N13}.

Here we will consider the following subsolution problem for the complex Hessian equation.  

\smallskip

{\it The continuous subsolution problem :}
 Let $\mu$ be a positive Borel measure on $\Omega$.  Assume that there exists a  function $\varphi \in  \mathcal{SH}_m (\Omega) \cap \mathcal{C}^{0} (\bar{\Omega})$ satisfying the following conditions :
\begin{equation} \label{eq:subsolution}
\mu \leq (dd^c \varphi)^m \wedge \beta^{n - m}, \, \, \mathrm{on} \, \, \, \, \Omega,  
\, \, \, \mathrm{and} \, \, \varphi_{\mid  \partial \Omega} \equiv 0.
\end{equation}

$(i)$  Does the  Dirichlet problem (\ref{eq:DirPb}) admit a continuous  solution $U_{g,\mu} \in \mathcal{SH}_m (\Omega) \cap \mathcal{C}^{0} (\bar{\Omega})$ for any continuous boundary datum $g$? 
 
$(ii)$ Is it possible to estimate  the modulus of continuity of the solution $U_{g,\mu}$ in terms of the modulus of continuity of  $\varphi$ and $g$ and some characteristic function related to $\mu$ ?

\smallskip

The continuous subsolution problem has attracted a lot of attention these last years. It was formulated 
by Ko\l odziej  for the complex Monge-Amp\`ere equation and  addressed in \cite{DGZ16} in the case of the existence of H\"older continuous subsolution. 
 The H\"older continuous subsolution problem was solved  recently for positive Borel measures with finite mass in \cite{BZ20} . 

Recently S. Ko\l odziej and  N.C. Nguyen gave a  Dini type sufficient condition on the modulus of continuity of the subsolution which guarantees the existence of a continuous solution for the complex Monge-Amp\`ere equation and  Hessian equations under the restrictive assumption that the measure $\mu$ is compactly supported in $\Omega$ (see \cite{KN20b}, \cite{KN20a})

The main goal of this paper is  to generalize these results by removing the assumption on the compactness of the support and to extend  them to the complex Hessian equation using an original idea from \cite{KN20b} . Moreover we are able to improve the quantitative estimate on the modulus of continuity of the solution obtain in the Hölder continuous case by  \cite{BZ20} and  also those of \cite{KN20b}, \cite{KN20a} in the  case of a compactly supported measure.
\subsection{Main results} 
 Our first main result gives a sufficient condition on the Borel measure $\mu$ in terms of its diffusion with respect to the $m$-Hessian capacity which guarantees the existence of a continuous solution to the Dirichlet problem (\ref{eq:DirPb}).

  \smallskip
 \smallskip
 
{\bf Theorem 1}. {\it  Let  $\Omega \Subset \C^n$ be a bounded strictly $m$-pseudoconvex  domain and $\mu$ be a positive Borel measure on $\Omega$ with finite mass. Assume that  there exists a constant $A > 0$ such that for any compact set $K \subset \Omega$,
$$
\mu (K) \leq A \text{c}_m (K) \gamma (\text{c}_m (K)),
$$
where $\gamma : \R^+ \longrightarrow \R^+$ is a continuous increasing function on $\R^+$ which  satisfies  the following  Dini type condition
 \begin{equation} \label{eq:DiniConditionMu}
\int_{0^+} \frac{\gamma (t)^{1 \slash m}}{t} d t < + \infty.
 \end{equation}
 Then   for any    continuous boundary datum $g \in \mathcal C^0 (\partial \Omega)$, the Dirichlet problem (\ref{eq:DirPb}) admits a unique solution  $U = U_{\mu,g} \in \mathcal{SH}_m (\Omega) \cap \mathcal C^0  (\bar{\Omega})$.}

The capacity $\text{c}_m (K) = \text{c}_m (K,\Omega)$ will be defined in the next section. 

Our second main result gives a new comparison inequality which will be applied to  positive Borel measures without restriction on their support nor on their mass. 

Let us fix $0 < r < m \slash (n-m)$ and $0 < b < 2 n$ and define the following functions for $t \in \R^+$:
\begin{equation}\label{eq:estimatefunction}
\ell_m (t) := \left\{\begin{array}{lcl} 
 t^{r}, \,  \, \, \, \, \, \, \, \, \, \,   \, \, \, \,  \, \, \, \, \, \, \, \, \, \,    \,  \, \, \, \, \, \, \hbox{if} \, \,  \, \, \, 1 \leq m < n, \\
   \exp ( -b \, t^{- 1 \slash n}), \, \, \,  \hbox{if} \, \, \, \, \,  m = n.
\end{array}\right.
\end{equation}
 
 \smallskip
 
 {\bf Theorem 2}.{ \it Let $\Omega \Subset \C^n$ be a bounded $m$-hyperconvex  domain and $\varphi \in \mathcal{SH}_m (\Omega)  \cap \mathcal{C}^0 (\bar{\Omega})$ with $\varphi = 0$ on $\partial \Omega$. 
 
 Then  there exists a constant $B = B (m,n, \varphi,\Omega) > 0$ such that for any compact set $K \subset \Omega$, 
$$
\int_{K}(dd^c\varphi)^m\wedge\beta^{n-m}  \leq  B \, \left\{\vartheta_m  (c_m (K)) + \left[\vartheta_m  (c_m (K))\right]^m\right\} \, c_m (K),
$$
where
$\vartheta_m (t) := \kappa_\varphi \circ \theta_m \circ \ell_m (3^m t)$ , $\kappa_\varphi$ is the modulus of continuity of $\varphi$ and $\theta_m$ is an  inverse  of the function $t \longmapsto t^{2m} \kappa_\varphi (t)^{1 - m}$ . }

 \smallskip
 \smallskip
  The constant $B$ in the theorem is given explicitly by the formula (\ref{eq:finalConst}).
 
 Theorem 2 extends an estimate obtained by S. Kolodziej and C. Nguyen for measures with compact support \cite{KN20b}. It improves the estimate proved in \cite{BZ20} in the H\"older continuous case.

 \smallskip
 \smallskip

 As a consequence of Theorem 1 and Theorem 2, we will deduce the following two results which solves the continuous subsolution problem under a Dini type condition on the modulus of continuity of the subsolution.

Since the two results are different for complex Monge-Amp\`ere equations and Hessian equations, we will state them separately.
 \smallskip
 
 \smallskip
  {\bf Theorem 3}.{ \it Let $\Omega \Subset \C^n$ be a  bounded strictly $m$-pseudoconvex  domain  with $1 \leq  m < n$ and $\mu $ a positive Borel measure on $\Omega$. 
  
  Assume that there exists $\varphi\in \mathcal{SH}_m(\Omega)\cap\mathcal{C}^{0}(\overline\Omega)$ such that  
\begin{equation} \label{eq:subsol2}
 \mu \leq (dd^c\varphi)^m\wedge\beta^{n-m}, \, \, \, \mathrm{weakly \, \,  on} \, \,  \Omega \, \, \, \mathrm{and} \, \, \varphi_{\mid  \partial \Omega} \equiv 0,
\end{equation} 
 and the modulus of continuity $\kappa_\varphi$ of $\varphi$ satisfies the following Dini type condition:
 \begin{equation} \label{eq:DC1}
 \int_{0^+} \frac{\left[\kappa_\varphi (t)\right]^{1 \slash m}}{t}  d t< + \infty, 
 \end{equation}

 Then for any continuous function $g \in \mathcal{C}^{0} (\partial \Omega)$, there exists a unique function $U = U_{g,\mu} \in \mathcal{SH}_m (\Omega) \cap \mathcal{C}^0 ({\Omega})$ such that  
 $$
 (dd^c U)^m\wedge\beta^{n-m} = \mu, \, \, \, \mathrm{weakly \, \,  on} \, \,  \Omega \, \, \, \mathrm{and} \, \, \, U_{\mid  \partial \Omega}  = g.
 $$

Moreover if $\mu (\Omega) < + \infty$, the $\widehat{\kappa}$-modulus of continuity of $U$ satisfies the following esstimate
 
$$
\widehat{\kappa}_U (\delta) \leq C  \, \tilde{\kappa}_m (\delta),
$$ 
 where  $\tilde{\kappa}_m (\delta)$ is given by the equation (\ref{eq:MC-Solution}) and $C = C(m,n, \mu, \varphi, \Omega) > 0$ is a uniform constant.} 
 
  \smallskip
 \smallskip
 
 For complex Monge-Amp\`ere equations (the case $m = n$) we obtain a much better result.
   
 \smallskip
 \smallskip
 
 {\bf Theorem 4}. { \it Let $\Omega \Subset \C^n$ be a  bounded strictly pseudoconvex  domain and $\mu $ a positive Borel measure on $\Omega$. 
 
 Assume that there exists $\varphi\in \mathcal{PSH} (\Omega)\cap\mathcal{C}^{0}(\overline\Omega)$ such that  
\begin{equation} \label{eq:subsol1}
 \mu \leq (dd^c\varphi)^n\, \, \, \, \mathrm{weakly \, \, on} \, \,  \Omega \, \, \, \mathrm{and} \, \, \varphi_{\mid  \partial \Omega} \equiv 0,
\end{equation} 
 and  the modulus of continuity $\kappa_\varphi$ of $\varphi$ satisfies the following Dini type condition:
 \begin{equation} \label{eq:DC2}
 \int_{0^+} \frac{\left[\kappa_\varphi (t)\right]^{1 \slash n}}{t \vert \log t \vert} d t< + \infty. 
 \end{equation}

 Then for any continuous function $g \in \mathcal{C}^{0} (\partial \Omega)$, there exists a unique function $U = U_{g,\mu} \in \mathcal{PSH} (\Omega) \cap  \mathcal{C}^0 ({\Omega})$ such that  
 $$
 (dd^c U)^n  = \mu, \, \, \, \mathrm{weakly \, \, on} \, \,  \Omega, \, \,  \, \, \, \mathrm{and} \, \, \, U_{\mid  \partial \Omega}  = g.
 $$

 Moreover if $\mu (\Omega) < + \infty$, the $\widehat{\kappa}$-modulus of continuity of $U$ satisfies the following estimate 
$$
\widehat{\kappa}_U (\delta) \leq C  \, \tilde{\kappa}_n (\delta),
$$ 
where  $\tilde{\kappa}_n (\delta)$ is given by the equation (\ref{eq:MC-Solution}) with $m = n$ and $C = C(n, \mu,\varphi,\Omega) > 0$ is a uniform constant.}

 \smallskip
 \smallskip
 
  Here the $\widehat{\kappa}$-modulus of continuity of  a given function $\phi:\Omega \longrightarrow \R$ is defined for $0 < \delta <\delta_0$  by
 \begin{equation}
\widehat{\kappa}_{\phi} (\delta) := \sup_{z \in \Omega_{\delta}} (\widehat{\phi}_\delta (z) - \phi (z)), \, \, \, 
\end{equation}
where
\begin{equation} 
\widehat{\phi}_\delta (z) := \int_{\B} \phi (z + \delta \xi) d \lambda_{\B} (\xi), 
\end{equation} 
for $z \in \Omega_\delta := \{z \in \Omega ; \text{dist} (z,\partial \Omega) > \delta \}$ and $0 < \delta < \delta_0$.
   
\smallskip 
 
 The first part of Theorem 4 was proved by S. Koldziej and C. Nguyen compactly supported  positive Borel measures (see \cite{KN20a}).

 \smallskip
 \smallskip
    
 Let us mention that  it is possible to give a uniform control on the full modulus of continuity $\kappa_U$ of the solution $U$ in Theorem 3 and Theorem 4 under some extra condition on the modulus of continuity of the data $g$ and the subsolution $\varphi$ (see Lemma \ref{lem:sup-mean}). Also when $g \in C^{1,1} (\partial \Omega)$ (e.g. $g\mid_{ \partial \Omega} \equiv 0$),  we can improve the estimate on the $\widehat{\kappa}$-modulus of continuity of the solution $U$  (see Theorem \ref{thm:ModC+}).  
 
 \subsection{Organization of the paper}
 
In section 1, we give  the necessary definitions and preliminaries that will be needed in the sequel. 

 \smallskip
 
Section 2 contains some new results which will play a crucial role in the proofs of our results. We first give a new estimate on the behaviour near the boundary of a domain of the $m$-Hessian measure of a continuous $m$-potential in terms of its modulus of continuity and the $m$-Hessian capacity on the domain. Then  we prove  continuity properties of these measures acting on normalized potentials.

 \smallskip
 
 Section 3 contains the proof of Theorem 1. We first give a priori uniform estimates and then prove continuity of the Hessian potentials of  Borel measures which are diffuse with respect to the corresponding capacity (see definition 3.1). Then we  establish a new stability estimate that improves the one obtained  in \cite{BZ20} under a weaker domination condition on the measure. This estimate is inspired by  an estimate proved  in \cite{BGZ08} in the spirit of \cite{EGZ09} in the case of compact K\"ahler manifolds.

 \smallskip
 
Section 4 contains the proof of Theorem 2 as well as some consequences.  The proof of this theorem  consists in extending a similar result proved in \cite{BZ20} in the H\"older continuous case.

 \smallskip
 
Section 5 contains the  proofs of Theorem 3 and Theorem 4. These proofs are done at the same time  in several steps following the same scheme. We first use the weak stabilty result Theorem \ref{thm:stability} to reduce the estimation of the modulus of continuity of the solution to the Dirichlet problem  (\ref{eq:DirPb}) to the estimate of the $L^m$-norm of the difference of two normalized potentials with respect to the measure $\mu$ using its domination by  the Hessian measure of the subsolution. Then we use results from Section 3 to estimate the  $L^m$-norm with respect to $\mu$ in terms of the $L^m$-norm with respect to the Lebesgue measure following a scheme which has become standard and which was initiated in \cite{EGZ09} and completed in  \cite{GKZ08} (see also \cite{DDGKPZ15} and \cite{GZ17}).

\section{Preliminary results}
 In this section, we recall the basic properties of $m-$subharmonic functions and some known results we will use  throughout the paper. 
 
\subsection{Hessian potentials}
 For a hermitian $n \times n$ matrix $a = (a_{j,\bar k})$ with complex coefficients, we denote by $\lambda_1, \cdots \lambda_n$ the eigenvalues of the matrix $a$. For any $1 \leq k \leq n$ we define the $k$-th trace of $a$ by the formula

$$
S_k (a) := \sum_{1 \leq j_1 < \cdots < j_k \leq n} \lambda_{j_1} \cdots \lambda_{j_k},
$$
which is the $k$-th elementary symmetric polynomial of the eigenvalues $(\lambda_1, \cdots, \lambda_n)$ of $a$.

 Recall that $d = \partial + \bar{\partial}$ and define $d^c := (i\slash 2) (\bar{\partial} - \partial$ so that $dd^c =  i \partial \bar{\partial}$ and denote by 
 $$
 \beta := dd^c \vert z\vert^2
 $$
  the standard K\"ahler form on $\C^n$.
 
 Let $\C^n_{(1,1)} $ be the space of real $(1, 1)$-forms on $\C^n$  with constant
coefficients, and define the cone of $m$-positive  $(1,1)$-forms on $\C^n$ by

$$
\Theta_m := \{\omega \in \C^n_{(1,1)}  \, ; \,  \omega \wedge  \beta^{n - 1} \geq 0, \cdots,  \omega^m \wedge  \beta^{n - m} \geq 0\}.
$$

\begin{definition}
1) A smooth $(1,1)$-form $\omega$ on $\Omega$ is said to be $m$-positive on $\Omega$ if for any $z \in \Omega$, $\omega (z) \in \Theta_m$.

2) A function $u:\Omega \rightarrow \mathbb{R}\cup\{-\infty\}$ is said to be  $m-$subharmonic  on $\Omega$ if it is subharmonic on $\Omega$ (not identically $-\infty$ on any component) and  for any collection of smooth $m-$positive $(1,1)-$forms  $\omega_1,...,\omega_{m-1}$ on $\Omega$, the following inequality holds in the sense of currents 
  $$
  dd^c u\wedge \omega_1\wedge...\wedge \omega_{m-1} \wedge \beta^{n-m}\geq 0,
  $$
  in the sense of currents on $\Omega$.
\end{definition}

We denote by $\mathcal{SH}_m (\Omega) $ the positive convex cone of $m$-subharmonic functions on $\Omega$ which are not identically $-\infty$ on any component of $\Omega$.  These are the $m$-Hessian potentials.

We give below the most basic properties of $m$-subharmonic functions that will be used in the sequel (see \cite{Bl05}, \cite{Lu12}).

\begin{proposition}\label{prop:basic}

\noindent 1.  If $u\in \mathcal{C}^2(\Omega)$, then $u$ is  $m$-subharmonic on $\Omega$ if and only if $(dd^c u)^k\wedge \beta^{n-k}\geq0$
pointwise on $\Omega$ for $k=1, \cdots, m$.

 \noindent 2. $\mathcal{PSH}(\Omega)=\mathcal{SH}_n(\Omega)\subsetneq \mathcal{SH}_{n-1}(\Omega)\subsetneq...\subsetneq \mathcal{SH}_1(\Omega)=\mathcal{SH}(\Omega) $.
 
\noindent 3.  $\mathcal{SH}_m(\Omega) \subset L^1_{loc} (\Omega)$ is a positive convex cone. 
  
\noindent 4.  If $u$ is $m$-subharmonic on $\Omega$ and $f: I \rightarrow\mathbb{R}$ is a  convex, increasing function on some interval containing the image of $u$, then $f \circ u$ is $m$-subharmonic on $\Omega$.

\noindent 5. The limit of a decreasing sequence of  functions in $\mathcal{SH}_m(\Omega)$ is $m$-subharmonic on $\Omega$ when it is not identically $- \infty$ on any component.
 
\noindent 6.  Let $u$ be an $m$-subharmonic function on $\Omega$. Let $v$ be an $m$-subharmonic function on a domain $\Omega'  \subset \C^n$ with $\Omega \cap \Omega' \neq \emptyset$. If $u\geq v$ on $\Omega \cap \partial\Omega'$, then the function
   $$
   z \mapsto w(z):=\left\{\begin{array}{lcl}
\max(u(z),v(z)) &\hbox{ if}\ z \in \Omega \cap \Omega'\\
 u(z)  &\hbox{if}\  z \in\Omega\setminus\Omega'\\
\end{array}\right.
$$
is $m$-subharmonic on $\Omega$.

 \end{proposition}

 \subsection{Approximation of Hessian potentials}
 Another ingredient which will be important is the regularization process.  
 
 Let $\Omega \Subset \C^n$ be a bounded domain. Set
$$
\delta_0 := \max\{d (z , \partial\Omega) \, ; \, z \in  {\bar \Omega} \}\cdot
$$   
  Then for any $ 0 < \delta < \delta_0$,  we have
  $$ 
  \Omega_\delta :\{z \in \Omega \, ; \, d (z , \partial\Omega) > \delta \} \neq \emptyset.
  $$  
   Fix a  non negative  radial Borel  function $\chi$ on $\C^n$ with compact support in the unit ball $\B \subset \C^n$ such that $\int_{\C^n} \chi (\zeta) d \lambda_{2 n} (\zeta) = 1$ and set for $ \delta > 0$, 
   $\chi_{\delta}(\zeta)=\frac{1}{\delta^{2n}}\chi (\frac{\zeta}{\delta})$.
 
 Let $u \in \mathcal{SH}_m (\Omega) \subset L^1_{loc} (\Omega)$ and define its standard $\delta$-regularization by the formula
      
  \begin{equation} \label{eq:reg}
  u_\delta (z) =  u \star {\chi}_\delta  (z) := \int_{\Omega} u (z-\zeta) \chi_{\delta} (\zeta) d \lambda_{2n} (\zeta), z \in \Omega_\delta.
  \end{equation}
  Then it is easy to see that $ {u}_{\delta}$ is $m$-subharmonic  on $\Omega_{\delta}$ and  decreases to $u$ in $\Omega$ as $\delta $ decreases to $0$.
  
  Observe that when $\chi= \chi^\B  := (1/\tau_{2n})  {\bf 1}_\B$ is the normalized characteristic function of the unit ball, then $u \star \chi_\delta = \widehat{u}_\delta$ is the mean-value function of $u$  defined on $\Omega_\delta$  by
  $$
   \widehat{u}_\delta (z) := (1\slash \tau_n) \int_\B u (z +\delta \zeta) d \lambda_{2n} (\zeta), \, \, z \in \Omega_\delta,
  $$
where $\tau_n := \lambda_{2n} (\B)$.
\begin{lemma}   \label{lem:Poisson-Jensen}
Let $u \in \mathcal{SH}_m (\Omega) \cap L^{1} (\Omega)$. Then for $0 < \delta < \delta_0$, its $\delta$-regularization extends to $\C^n$ by the formula
      
  \begin{equation} \label{eq:reg}
 u_\delta (z) =  u \star  {\chi}_\delta (z)  := \int_{\Omega} u (\zeta) \chi_{\delta} (z - \zeta) d \lambda_{2n} (\zeta), z \in \C^n,
  \end{equation}
and have the following properties :
 
 1)  the function $ { u}_{\delta}$  is $m$-subharmonic in $\Omega_\delta$, smooth in $\C^n$ provided that $\chi$ is smooth;

 2) $({u}_\delta)$ decreases to $u$ in $\Omega_{\delta_0}$ as $\delta \in ]0,\delta_0[ $ decreases to $0$, where $\delta_0 > $ is small enough;

 3) the mean-value function $\widehat{u}_\delta $ satisfies the estimate
  \begin{equation} \label{eq:PJ1}
  \int_{\Omega_\delta} \left(\widehat{u}_\delta (z) - u(z)\right) d\lambda_{2 n}(z) \leq a_n \delta^2 \int_{\Omega_\delta} dd^c u \wedge \beta^{n - 1},
  \end{equation}
  where $a_n > 0$ is a  constant which does not depend on $u$ nor on $\delta$.
  
 4)  If $\Omega$ is  $m$-hyperconvex (see Definition 1.4) with smooth boundary and $u \leq 0$ on $\Omega$, we have for any $0 < \delta< \delta_0$,
 \begin{equation} \label{eq:PJ2}
 \int_{\Omega_\delta} ( {u}_{\delta} (z) - u(z)) d\lambda_{2 n}(z) \leq b_n \delta \Vert u\Vert_1,
    \end{equation}
  where $\Vert u\Vert_1 := \int_{\Omega} \vert u\vert d \lambda_{2 n}$ and $b_n > 0$ is a uniform constant which does not depend on $u$ nor on $\delta$.
\end{lemma}
An estimate like  (\ref{eq:PJ2}) was first obtained in  \cite{BKPZ16} (see also \cite{KN20a} and \cite{Ze20})).
\begin{proof}
The first and the second property are clear.  The third one follows from Poisson-Jensen formula for subharmonic functions (see \cite{GKZ08}, \cite{Ze20}).

To prove the last property, observe that by definition there is a defining function $\rho$ which is $m$-subharmonic and smooth on $\bar \Omega$ and satisfies  $\vert \nabla \rho\vert >0$ on $\partial \Omega$. Then there exists a constant $c > 0$ depending only on $\Omega$ such that   $- \rho (z) \geq c \, \mathrm{dist} (z,\partial \Omega)$ for any $z \in \Omega$ (see  \cite{Ze20} for more details). Hence
$\Omega_\delta \subset \{\rho < - c \delta\}$ and then by formula (\ref{eq:testinequality}) we get
\begin{eqnarray*}
\int_{\Omega_\delta} dd^c u \wedge \beta^{n - 1} &\leq & \int_\Omega \frac{1}{ c \delta} (- \rho)   dd^c u \wedge \beta^{n - 1} \\
&\leq  &\frac{1}{ c \delta} \int_\Omega (- u)   dd^c \rho \wedge \beta^{n - 1}\leq \frac{b_n}{\delta}  \Vert u\Vert_1,
\end{eqnarray*}
where $b_n >0$ is a uniform constant.
\end{proof}

Let us introduce  the notions of $m$-pseudoconvexity.

\begin{definition} \label{eq:HyperConv} 1.  We say that the open set $\Omega \Subset \C^n$  is  $m$-hyperconvex  if it admits a defining function $\rho : \Omega \longrightarrow ]-\infty , 0[ = \R_{>0}$ which is a bounded continuous $m$-subharmonic on $\Omega$ (see \cite{Lu12,Lu15}. If moreover $\rho$ is smooth in $\bar \Omega$ and $\nabla \rho\vert > 0$ pointiwse in $\partial \Omega$, we will say that $\Omega \Subset \C^n$  is  $m$-hyperconvex with smooth boundary.

2. We say that the open set $\Omega \Subset \C^n$  is  strictly $m$-pseudoconvex if $\Omega$ admits a smooth defining function $\rho$ which is strictly $m$-subharmonic in a neighborhood of $\bar \Omega$ and satisfies  $\vert \nabla \rho \vert > 0$ pointwise  in $\partial \Omega = \{\rho = 0\}$. In this case we can choose $\rho $ so that
\begin{equation} \label{eq:stronpconvexity}
(dd^c \rho)^k \wedge \beta^{n - k} \geq \beta^n \, \, \mathrm{for} \, \, 1 \leq k \leq m,
\end{equation}
pointwise in $\Omega$.
\end{definition}

 \begin{example}
1.  Any euclidean ball in $\C^n$ is strictly  $m$-pseudoconvex and any polydisc in $\C^n$ ($n \geq 2$) is $m$-hyperconvex but not strictly  $m$-pseudoconvex.

2. The domain $\{ z \in \C^n ; \sum_{1 \leq j \leq n} \vert z_j\vert < 1 \}$ is a bounded  $m$-hyperconvex domain with Lipschitz but not smooth boundary, hence it is not strictly  $m$-pseudoconvex. 
\end{example}

The following lemma will be also needed. 
 For a function $g \in \mathcal{C}^0 (\partial \Omega)$, we denote by   $\mathcal{SH}^g_m (\Omega)$ the set of functions $w \in \mathcal{SH}_m (\Omega) \cap L^{\infty} (\Omega)$ such that $w =g$ on $\partial \Omega$ i.e.  for any $\zeta \in \Omega$, $\lim_{z \to \zeta} w (z) =  g (\zeta)$.
 
\begin{lemma} \label{lem:appximationwithbdv} Let $g \in \mathcal{C}^0 (\partial \Omega)$ and   $w \in \mathcal{SH}_m^g (\Omega)$. Then there exists a decreasing sequence $(w_j)$ of functions in $\mathcal{SH}_m^g  (\Omega) \cap C^0 (\bar{\Omega})$ which converges to $w$ pointwise on $\Omega$.
\end{lemma}
\begin{proof} First take any decreasing sequence  of continuous functions $(h_j)$ on $\bar{\Omega}$ which converges to $w$ on $\bar{\Omega}$. We can arrange so that $h_j = g$ on $\partial \Omega$. Indeed take the harmonic extension $G$ of $g$ to $\Omega$ and then the sequence $\min \{h_j,G\}$ satisfies the requirement. 

Now  set
$$
w_j := \sup \{v \in \mathcal{SH}_m (\Omega) ; v \leq h_j\}.
$$ 
 By \cite{BZ20}, we know that the sequence $(w_j)$ satisfies all the requirements of the lemma.
\end{proof}

\subsection{Remarks on the modulus of continuity}

 Let  $\phi : {\Omega} \longrightarrow \R$ be a continuous function.  We fix $\delta_0 > 0$ as before so that $\Omega_{\delta_0} \neq \emptyset$ and recall the following definition for $0 < \delta < \delta_0$ and $z \in \Omega_\delta$,
\begin{equation} \label{eq:supnorm}
\widehat{\phi}_\delta  (z) := \int_{\bar{\B}} \phi (z + \delta \zeta) d \lambda_\B (\zeta),  \, \, \, 
\end{equation}
where $\lambda_\B$ is the normalized Lebesgue measure on $\B$.

 We introduce the  modulus of (uniform)  continuity of $\phi$ on ${\Omega}$ defined for $\varepsilon> 0$ by the formula
  \begin{equation} \label{eq:fullkappa}
  \kappa_\phi (\varepsilon) := \sup \{\vert \phi(z) - \phi (z')\vert \, ; \, z, z' \in {\Omega}, \vert z-z'\vert \leq \varepsilon\}.
  \end{equation}
  Then $\phi$ extends to a uniformly continuous function on $\bar{\Omega}$ if and only if $\lim_{\varepsilon \to 0^+} \kappa_\varphi (\varepsilon) = 0$.
  
We introduce another  modulus of continuity defined  for $0 < \delta < \delta_0$ by the formula
\begin{equation} \label{eq:hatkappa}
\widehat{\kappa}_\phi (\delta) := \sup_{\Omega_\delta} \left(\widehat{\phi}_\delta (z) - \phi (z)\right).
\end{equation}

We see immediately that $\widehat{\kappa}_\phi (\delta) \leq  \kappa_\phi (\delta)$ for any $0< \delta < \delta_0$. 

These  moduli quantify the continuity of $\phi$ on $\Omega$. While the (full) modulus of  continuity $ \kappa_\phi$ characterizes uniform continuity of $\phi$ on $\bar{\Omega}$, the (relative) modulus of continuity $\widehat{\kappa}_\phi$ only characterizes the continuity of $\phi$ on $\Omega$. Indeed  the condition $\lim_{\delta \to 0^+} \widehat{\kappa}_\phi (\delta) = 0$  implies that the function $\phi$ is continuous on $\Omega$, but it does not imply the extension of the function $\phi$ by continuity to $\bar{\Omega}$ as the example of a harmonic function on $\Omega$ shows. 

We will state a result from \cite{Ze20}  which clarifies the relations between these notions of continuity in some cases.

We need some definitions. 
\begin{definition} 1.  A continuous  function $\kappa : [0, l] : \longrightarrow \R^+$ is a modulus of continuity if it is  increasing,  subadditive and satisfies $\kappa (0) =0$. It's always possible to extend such a function to the whole $\R^+$ with the same properties.

2. A function  $\phi : \Omega \longrightarrow \R$ is said to be $\kappa$-continuous near the boundary $\partial \Omega$ if there exists $0 < \delta_1 < \delta_0$ small enough and a constant $M_1 > 0$ such that for any $\zeta \in \partial \Omega$ and any $z \in \Omega$ with $\vert z - \zeta\vert \leq \delta \leq \delta_1$, we have $\vert u (z) - u(\zeta) \vert \leq M_1  \kappa (\delta)$.
\end{definition}
Uniform continuity on $\bar{\Omega}$ implies uniform continuity near the boundary $\partial \Omega$. However as observed above,  the condition $\lim_{\delta \to 0} \widehat{\kappa}_\phi (\delta) = 0$ implies the continuity of $\phi$ on $\Omega$ but it does not imply continuity near the boundary $\partial \Omega$.

We first introduce the following condition on $\kappa$.
\begin{equation} \label{eq:MCcondition}
\exists A > 1, \, \, \limsup_{t \to 0^+} \frac{\kappa (A t)}{A \kappa (t)} <  \frac{1}{2n}.
\end{equation}
Observe that  the condition (\ref{eq:MCcondition}) is satisfied by any logarithmic H\"older modulus of continuity $\kappa_{\alpha,\nu}  (t) := t^\alpha (-\log t)^{\nu}$ for $0 < t < < 1$ where $0\leq  \alpha < 1$ and $\nu \in \R$, with $\nu < 0$ when $\alpha =0$.

The following lemma  is proved in  \cite{Ze20}. 
\begin{lemma} \label{lem:sup-mean}  Let $\kappa$ be a modulus of continuity satisfying (\ref{eq:MCcondition}).  Let $\Omega \Subset \C^n$ be a bounded domain and $u \in  \mathcal{SH} (\Omega) \cap L^{\infty} ({\bar \Omega})$. Assume that $u$ is $\kappa$-continuous  near $\partial\Omega$. Then the following properties are equivalent:

$(i)$ $\exists c_1 >0$, $\exists \,  \delta_1$ with $ 0 < \delta_1 < \delta_0$ such that for any $0< \delta < \delta_1$
$$
\widehat{u}_\delta  (z)   \leq u  (z) +  c_1 \kappa (\delta), \, \, \, \text{for  any } \, \, \, z \in \Omega_\delta,
$$ 

$(ii)$ $\exists c_2 >0$, $\exists  \, \delta_2$ with $ 0 < \delta_2 < \delta_0$ such that for any $0< \delta < \delta_2$,
$$
\sup_{\bar B(z,\delta)} u \leq  u (z)+  c_2 \kappa (\delta),  \, \, \, \text{for  any } \, \, \, z \in \Omega_\delta.
$$

\smallskip

Moreover if one of these conditions is satisfied then $u$ is $\kappa$-continuous on $\bar{\Omega}$ i.e. $\kappa_u \leq C \, \kappa$, where $C > 0$ is a uniform constant.
\end{lemma}
 
\subsection{Complex Hessian operators}

Following \cite{Bl05}, we can define the Hessian operators acting on (locally) bounded $m$-subharmonic functions as follows. 
 Given $u_1, \cdots, u_k \in \mathcal{SH}_m (\Omega) \cap L^{\infty} (\Omega)$ ($1 \leq k \leq m$), one can define inductively the following  positive $(m-k,m-k)$-current on $\Omega$
$$
dd^c u_1 \wedge \cdots \wedge dd^c u_k \wedge \beta^{n - m} := dd^c (u_1 dd^c u_2 \wedge \cdots \wedge dd^c u_k \wedge \beta^{n - m}).
$$

In particular, if $u  \in \mathcal{SH}_m (\Omega) \cap L^{\infty}_{loc} (\Omega)$, the positive current $(dd^c u)^m \wedge \beta^{n-m}$ is of top degree and can be written as :
$$
 (dd^c u)^m \wedge \beta^{n-m} = \frac{m! (n-m)!}{n!} \sigma_m (u) \beta^n,
$$
where $ \sigma_m (u)$ is a  positive Borel measure on $\Omega$, called $m$-Hessian measure of $u$.
 
Observe that when $m= 1$,  $\sigma_1 (u) = dd^c u \wedge \beta^{n-1} \slash \beta^n $ is the Riesz measure of $u$ (up to a positive constant), while  $\sigma_n (u) = (dd^c u)^n \slash \beta^n$ is the complex   Monge-Amp\`ere measure of $u$ on $\Omega$. 

It is then possible to extend Bedford-Taylor theory to this context. 
In particular, Chern-Levine Nirenberg inequalities hold and the Hessian operators are continuous under local uniform convergence and  monotone convergence 
pointwise a.e. on $\Omega$ of sequences of functions in  $\mathcal{SH}_m (\Omega) \cap L^{\infty}_{loc} (\Omega)$ (see \cite{Bl05}, \cite{Lu12}).

We define  $\mathcal{E}_m^0 (\Omega) $ to be the positive convex cone of  negative  functions  $\phi \in \mathcal{SH}_m (\Omega) \cap L^{\infty} (\Omega)$ such that
$$
\int_{\Omega} (dd^c \phi)^m \wedge \beta^{n - m} < + \infty, \, \, \phi_{|\partial \Omega} \equiv 0.
$$
These are the "test functions" in  $m$-Hessian Potential Theory in the sense that Stokes theorem is valid for these functions (see \cite{Lu12}).

More precisely it follows from \cite{Lu12,Lu15} that the following property holds:  if $\phi \in  \mathcal{E}_m^0 (\Omega) $ and $u , v \in \mathcal{SH}_m (\Omega) \cap L^{\infty} (\Omega)$ with $u \leq 0$, then for $0 \leq k \leq m - 1$, 

\begin{equation} \label{eq:testinequality}
\int_\Omega (-\phi)  dd^c u \wedge (dd^c v)^k \wedge \beta^{n - k-1} \leq \int_\Omega (-u)  dd^c \phi \wedge (dd^c v)^k \wedge \beta^{n - k-1}. 
\end{equation}

An important tool in the corresponding Potential Theory is the Comparison Principle.

\begin{proposition} \label{prop:Comparison Principle}
 Assume that $u,v\in \mathcal{SH}_m(\Omega)\cap L^{\infty}(\Omega)$ and for any $\zeta \in \partial \Omega$, $\liminf_{z \rightarrow \zeta }(u(z)- v(z))\geq 0$.  Then 
 $$
 \int_{\{u<v\}}(dd^c v)^m\wedge\beta^{n-m} \leq \int_{\{u<v\}}(dd^c u)^m\wedge\beta^{n-m}.
 $$
 Consequently, if $(dd^cu)^m\wedge\beta^{n-m}\leq(dd^cv)^m\wedge\beta^{n-m}$ weakly on $\Omega$, then $u \geq v$ in $\Omega$.
 
\end{proposition}
It follows from the comparison principle that if the Dirichlet problem (\ref{eq:DirPb}) admits a solution, then it is unique.

Let us recall the following estimates due to Cegrell (\cite{Ceg04}) for complex Monge-Amp\`ere operators and extended by Charabati to  complex Hessian operators (\cite{Ch16a}).

\begin{lemma}   \label{lem:Cegrell} 
Let  $u, v, w \in\mathcal{E}_m^0(\Omega)$. Then for any $1 \leq k \leq m - 1$
   
 $$
    \begin{array}{lcl}
 \int_{\Omega}dd^cu\wedge(dd^cv)^k\wedge(dd^cw)^{m-k-1}\wedge\beta^{n-m}
     \leq  I_m (u)^{\frac{1}{m}} \, I_m (v)^{\frac{k}{m}} \,  I_m (w)^{\frac{m-k-1}{m}},
  \end{array}
 $$
 where $I_m (u) :=  \int_{\Omega}(dd^c u)^m \wedge \beta^{n-m}$.
 
 In particular, if $\Omega$ is strictly $m$-hyperconvex, then
 $$
 \int_{\Omega}dd^c u \wedge (dd^c w)^k \wedge\beta^{n-k -1}  \leq  c_{m,n} \left(I_m (u)\right)^{\frac{1}{m}} \left(I_m (w)\right)^{\frac{k}{m}},
 $$
 and 
$$ 
  \int_{\Omega}dd^c u  \wedge \beta^{n-1}  \leq  c_{m,n} \left(I_m (u)\right)^{\frac{1}{m}}
 $$
 where $c_{m,n} > 0$ is a uniform constant.
\end{lemma}

The following result  will be useful (see \cite{BZ20}).

\begin{lemma} \label{lem:Comparison Principle} Let $\Omega \Subset \C^n$ be a bounded strictly $m$-pseudoconvex domain. Assume that  $u,v\in \mathcal{SH}_m(\Omega)\cap L^{\infty}(\Omega)$ satisfy $u \leq v$ on $\Omega$ and for any $\zeta \in \partial \Omega$, $\lim_{z \rightarrow \zeta }(u(z)- v(z))= 0$.
Then
 $$
 \int_{\Omega}(dd^c v)^m\wedge\beta^{n-m} \leq \int_{\Omega}(dd^c u)^m\wedge\beta^{n-m}.
 $$
\end{lemma}

 We will also need the following result which was proved by B\l ocki for the complex Monge-Amp\`ere operator (see \cite{Bl93})
 \begin{lemma} \label{lem:Blocki} Let $\psi,  v, w \in \mathcal{SH}_m (\Omega) \cap L^{\infty} (\Omega)$ such that $\psi \leq 0$,  $v \leq w$ and 
$ \lim_{z \to \zeta} (w (z) - v(z)) = 0$. Then
$$
\int_\Omega (w-v)^m (dd^c \psi)^m \wedge \beta^{n-m} \leq m! \Vert \psi \Vert_{\infty}^m  \int_\Omega  (dd^c v)^m \wedge \beta^{n-m}
$$ 
 \end{lemma}
The proof in this case is the same as in \cite{Bl93} since it essentially only uses the integration by parts formula.

\section{Hessian measures of continuous potentials}

\subsection{Hessian capacities} 

An important tool in dealing with our problems is the notion of capacity. This was introduced by  Bedford and Taylor in their pioneer work for the complex Monge-Amp\`ere operator (see \cite{BT82}). 
Let us recall the corresponding notion of capacity we will use here (see \cite{Lu12}, \cite{SA13}). Let $\Omega \Subset \C^n$  be a $m$-hyperconvex domain.  The $m$-Hessian capacity  is defined as follows. For any compact set $K \subset \Omega$,
$$
 \text{c}_m(K,\Omega) := \sup \{\int_K  (dd^c u)^m \wedge \beta^{n - m} ; u \in \mathcal{SH}_m (\Omega) , - 1 \leq u \leq 0\}.
$$

We can extend this capacity as an outer capacity on $\Omega$. Given a  set $S \subset \Omega$, we define the inner capacity of $S$ by the formula
$$
\text{c}_m(S,\Omega) := \sup \{\text{c}_m(K,\Omega) ; K \, \, \hbox{compact} \, \, K \subset S\}.
$$ 

The outer capacity of $S$ is defined by the formula 
$$
\text{c}^*_m(S,\Omega) := \inf \{\text{c}_m(U,\Omega) ; U \, \, \hbox{ is open} \, \, U \supset S\}, 
$$ 

One can show that $\text{c}^*_m(\cdot,\Omega)$ is a Choquet capacity and then  any Borel set$ B \subset \Omega$ is capacitable and
for any compact set $K \subset \Omega$, 
 \begin{equation} \label{eq:cap}
 \text{c}^*_m(K,\Omega) = \text{c}_m(K,\Omega)=\int_{\Omega}(dd^c u_K^*)^m\wedge\beta^{n-m},
 \end{equation}
  where $u_K$ is the relative equilibrium potential of $(K,\Omega)$ defined by the formula :
  
 $$
 u_K:=\sup\{u\in \mathcal{SH}_m(\Omega);u\leq0\ \hbox{in}\ \Omega, u\leq-{\bf 1}_K  \, \, \mathrm{on } \, \, \Omega\},
 $$
 and $u_K^*$ is its upper semi-continuous regularization on $\Omega$ (see \cite{Lu12}).
 
 It is well knwon that $u_K^*$ is $m$-subharmonic on $\Omega$, $- 1 \leq u_K^* \leq 0$, $u_K^* = - 1$ quasi-everywhere (with respect to $ \text{c}_m$) in $K$ and $u_K^* (z) \to 0$ as $z \to \partial \Omega$ (see \cite{Lu12}).

\subsection{Hessian mass estimates near the boundary} 

Here we prove a comparison inequality which seems to be new even in the case of a complex  Monge-Amp\`ere measure. This will play a crucial role in the proof of Theorem 2  and may have an interest in its own. It is a generalization of an estimate proved in \cite{BZ20} for Hessian measures of H\"older continuous potentials.
\begin{lemma} \label{lem:ComparisonIneq} Let $\Omega \Subset \C^n$ be a $m$-hypercovex domain and 
 $\varphi \in \mathrm{SH}_m (\Omega) \cap L^{\infty}  (\Omega)$ such tha $\varphi < 0$ in $\Omega$. Then for any compact set $K \subset \Omega$ we have
$$
\int_K (dd^c \varphi)^m \wedge \beta^{n - m}  \leq  [\text{osc} (\varphi, K,\Omega)]^m  \, \text{c}_m (K,\Omega),
$$
where $\text{osc} (\varphi, K,\Omega) := \sup_{\Omega} \varphi - \inf_K \varphi$

If moreover $\varphi$ is continuous in $\bar \Omega$ and $\varphi = 0$ on $\partial \Omega$,  then for any compact subset $K \subset \Omega$, we have
 $$
\int_K (dd^c \varphi)^m \wedge \beta^{n - m} \leq \left[\kappa (\delta_K (\partial \Omega))\right]^m   \,  \text{c}_m (K,\Omega),
 $$
 where $\kappa = \kappa_\varphi$ is the  modulus of continuity of $\varphi$ on $\bar{\Omega}$ and 
 $$\delta_K (\partial \Omega) := \sup_{z \in K}  \mathrm{dist} (z,\partial \Omega)
 $$
  is the Hausdorff distance of $K$ to the boundary $\partial \Omega$.
\end{lemma}
It should be mentioned that the strength of this lemma lies in the fact that whenever K is a very fine set located very close to the boundary, its capacity becomes larger, so that there is an interaction between the size of $ K$ and its position near the boundary.
\begin{proof} 
1) We can assume that $\varphi $ is non constant in $\Omega$ and $\sup_{\Omega} \varphi = 0$. Then  $ a  := \text{osc} (\varphi, K,\Omega) = - \inf_K \varphi > 0 $ and the function $v := a^{-1} \varphi $  is  $m$-subharmonic in $\Omega$, and satisfies the inequalities  $v \leq 0$ in $\Omega$ and $v \geq - 1$ in $K$.



 Fix $\varepsilon >0$ and let $u_K$ be the relative extremal $m$-subharmonic function of $(K,\Omega)$.  Then $ K \subset \{ (1 + \varepsilon) u_K^*  < v\} \cup \{u_K < u_K^*\}$. Since the set $\{u_K < u_K^*\}$ has zero $m$-capacity (see \cite{Lu12}), it follows from  the comparison principle that
\begin{eqnarray*}
\int_K (dd^c v)^m \wedge \beta^{n - m}  &\leq &\int_{\{ (1 + \varepsilon) u_K^*  < v\}} (dd^c v)^m \wedge \beta^{n - m} \\
&\leq & (1 + \varepsilon)^m \int_{\{(1 + \varepsilon) u_K ^* < v\}} (dd^c u_K^*)^m \wedge \beta^{n - m} \\
& \leq & (1 + \varepsilon)^m \int_\Omega  (dd^c u_K^*)^m \wedge \beta^{n - m}  =  (1 + \varepsilon)^m \text{c}_m (K,\Omega).
\end{eqnarray*}
The last identity follows from \cite{Lu12}.  Letting $\varepsilon \to 0$, we obtain the first statement. 

2) Fix a compact set $K \subset \Omega$. Since $\kappa$ is the modulus of continuity of $\varphi$, we have for any $\zeta \in \partial \Omega$ and $z \in K$
$\varphi (\zeta) - \varphi (z) \leq \kappa (\vert \zeta - z\vert)$.  Since $\varphi = 0$ in $\partial \Omega$, we obtain that for any $z \in K$,
$$
- \varphi (z) \leq \kappa \left( \delta_K (\partial \Omega)\right).
$$
Hence $\sup_K (- \varphi) \leq   \kappa \left( \delta_K (\partial \Omega)\right)$.  Applying the first statement to $\varphi$ we obtain the required inequality. 
\end{proof}

\subsection{ Hessian measures acting on Hessian potentials} 
We will study continuity properties of the functional associated to the Hessian measure of a function  $\varphi\in \mathcal{SH}_m(\Omega)\cap \mathcal{C}^{0} (\bar{\Omega}) $, acting on the space $\mathcal{SH}_m (\Omega) \cap L^{\infty} (\Omega)$.

 Let $g \in \mathcal{C}^0 (\partial \Omega)$ be a continuous function on  $\partial \Omega$ and $R > 0$ a positive constant. We denote by $\mathcal{SH}_m^g (\Omega,R)$ the set of functions $v \in \mathcal{SH}_m (\Omega) \cap L^{\infty} (\Omega)$ such that 
 $$
 \int_\Omega (dd^c v)^m \wedge  \beta^{n-m} \leq R, \,   \, \, \text{and} \, \, \, \lim_{z \to \zeta} v (z) = g (\zeta), \, \, \forall \zeta \in \partial \Omega.
 $$
 
The following result improves previous estimates given in \cite{N14} and \cite{BZ20}. 
 \begin{theorem}\label{thm:ModC}
Let $\varphi\in \mathcal{SH}_m(\Omega)\cap \mathcal{C}^{0} (\bar{\Omega})$ and $g \in \mathcal{C}^0 (\partial \Omega)$ be given functions.
Then  there exists $C_m  = C (m,\Omega,g, R) >0$   such that for every $u,v\in \mathcal{SH}^g_m(\Omega, R)$,
we have
\begin{equation} \label{eq:MocEst}
\int_{\Omega}|u-v|^m (dd^c\varphi)^m\wedge\beta^{n-m}\leq C_m \, \kappa_\varphi \circ \theta_m \left(\Vert u-v \Vert_m^{m}\right),
\end{equation}
where $\Vert u-v \Vert_m := \left(\int_{\Omega} \vert u - v\vert^m d \lambda_{2 n} \right)^{1 \slash m}$,  $\theta_m$ is the reciprocal of the function $t \longmapsto t^{2m} \kappa_{\varphi}^{1-m} (t)$.
\end{theorem}

\begin{proof}
Observe that for any $\varepsilon > 0,$ $ u_\varepsilon := \max \{u - \varepsilon, v \} \in \mathcal{SH}^g_m(\Omega)$, $u_\varepsilon \geq v$ and $u_\varepsilon = v$ near the boundary $\partial \Omega$. By the comparison principle, this implies that $u_\varepsilon \in  \mathcal{SH}^g_m(\Omega, R)$. Therefore, replacing $u$ by $u_\varepsilon$,  we can assume that $u \geq v$ on $\Omega$ and $u = v$ near the boundary, for the inequality (\ref{eq:MocEst}) will follow since $\vert u - v \vert =   ( \max\{u,v\} - u) + (\max\{u,v\}  - v)$.

On the other hand by approximation on the support $S$ of $u-v$ which is compact, we can assume that $u$ and $ v $ are smooth on a neighbourhood of $S$.

Then it remains to estimate the following integral
$$
I_m := \int_{\Omega} (u - v)^m  (dd^c\varphi)^m\wedge\beta^{n-m}.
$$

 Observe  first that we can extend $\varphi$ by continuity to $\C^n$ with the same modulus of continuity. Indeed, it is easy to see that the function defined  for $z \in \C^n$ by the following formula 
  $$
  \tilde{\varphi} (z) := \sup \{ \varphi (\zeta) - \kappa_\varphi (\vert z - \zeta\vert) \, ; \; \zeta \in \bar{\Omega}\}\cdot
  $$
  is the required extension.
 For simplicity, it will be denote  by $\varphi$.   
 
 Then we denote by $\varphi_{\delta}$ the smooth approximants of $\varphi $ on $\C^n$, defined by (\ref{eq:reg}).
    
  We know  that for $0 < \delta < \delta_0$, $\varphi_\delta \in  \mathcal{SH}_m (\Omega_{\delta})\cap\mathcal{C}^{\infty}(\mathbb C^n)$.

  Since $\varphi_\delta$ is not $m$-subharmonic on the whole  $\Omega$, we will consider  its $m$-subharmonic envelope defined by the formula :
  
  \begin{equation} 
  \psi_\delta (z) := \sup \{\psi (z) \, ; \, \psi \in \mathcal{SH}_m (\Omega), \psi \leq \varphi_\delta \, \, \text{on} \,\, \Omega\}, \, \, z \in \Omega.
  \end{equation} 
  We know by \cite[Theorem 3.3 ]{BZ20} that $\psi_\delta \in \mathcal{SH}_m (\Omega) \cap \mathcal{C}^0 (\bar{\Omega}),$ $ \psi_\delta  \leq \varphi_\delta $ on $ \Omega$ and 
  \begin{equation} \label{eq:projection}
  (dd^c  \psi_\delta)^m \wedge \beta^{n-m} \leq (\sigma_m (\varphi_\delta))_+,
   \end{equation} 
  weakly on $\Omega$, where $(\sigma_m (\varphi_\delta))_+$ is defined pointwise on $\Omega$ by  $(\sigma_m (\varphi_\delta))_+(z) = \sigma_m (\varphi_\delta)(z)$ for $z \in  \Omega$ such that $dd^c \varphi_\delta (z) \in  \Theta_m$ and $(\sigma_m (\varphi_\delta))_+(z) = 0$ otherwise.

  To prove the required estimate, we write for $0 < \delta < \delta_0$
$$
I_m = A_m (\delta) + B_m (\delta),
$$
where
$$
A_m (\delta) := \int_{\Omega} (u - v)^m \left[ (dd^c\varphi)^m - (dd^c\psi_\delta)^m\right]\wedge\beta^{n-m}.
$$
and 
$$
B_m (\delta) := \int_{\Omega} (u - v)^m  (dd^c\psi_\delta)^m \wedge\beta^{n-m}.
$$

We estimate each term separately for fixed $0 < \delta < \delta_0$.

To estimate the first term, observe that
$$
\left( (dd^c\varphi)^m - (dd^c\psi_\delta)^m\right) \wedge  \beta^{n - m}= dd^c (\varphi - \psi_\delta) \wedge T,
$$
where $T := \sum_{j = 0}^{m -1} (dd^c \varphi)^j \wedge (dd^c \psi_\delta)^{m-j-1} \wedge\beta^{n - m}$.

Then
$$
A_m (\delta) = \int_{\Omega} (u - v)^m dd^c (\varphi - \psi_\delta - \kappa_\varphi(\delta)) \wedge T. 
$$
Integration by parts yields
$$
A_m (\delta) = \int_{\Omega}  (\psi_\delta - \varphi + \kappa_\varphi(\delta) ) \left[- dd^c (u - v)^m \right]\wedge T. 
$$
An easy computation shows that 
\begin{eqnarray} \label{eq:formalineq}
- dd^c (u - v)^m \wedge T & \leq &  m  (u - v)^{m-1} dd^c (v - u) \wedge T \\
&\leq & m  (u - v)^{m-1} dd^c v \wedge T, \nonumber
\end{eqnarray} 
in the sense of currents on $\Omega$.

 Observe that from the definition we have  $\psi_\delta \leq \varphi_\delta  \leq \varphi + \kappa_\varphi (\delta)$ on $\Omega$. On the other hand, since $ \varphi - \kappa_\varphi (\delta) \leq \varphi_\delta$ on $\Omega$, it follows
that $ \psi_\delta - \varphi  + \kappa_\varphi  (\delta) \geq 0 $ on $\Omega$.  
Combining the two estimates  we conclude that  $0 \leq \psi_\delta - \varphi  + \kappa_\varphi  (\delta) \leq 2 \kappa_\varphi (\delta)$, and then 
$$
A_m (\delta) \leq 2 m \, \kappa_\varphi  (\delta)  \int_{\Omega}  (u - v)^{m-1} dd^c v \wedge T.
$$
By definition of $T$ we have  
\begin{eqnarray*}
&& \int_{\Omega}  (u - v)^{m-1} dd^c v \wedge T \\
 &&=   \sum_{j = 0}^{m -1} \int_{\Omega}  (u - v)^{m-1} dd^c v \wedge  (dd^c \varphi)^j \wedge (dd^c \psi_\delta)^{m-j-1} \wedge\beta^{n - m}
\end{eqnarray*}

Observe that  if we write 
$$
  (dd^c \varphi)^j \wedge (dd^c \psi_\delta)^{m-j-1} \wedge\beta^{n - m} = dd^c w \wedge S_j, 
  $$ 
  where $w = \varphi$ or $w = \psi_\delta$,  then as before by integration by parts  using an inequality analogous to (\ref {eq:formalineq}) with $k$ intead of $m$ and $ dd^c v \wedge S_j$ instead of $T$, we obtain that  for $1 \leq k \leq m$, 
$$
\int_{\Omega}  (u - v)^{k} dd^c v \wedge  dd^c w \wedge S_j \leq k  \Vert w \Vert_{\Omega} \int_{\Omega}  (u - v)^{k-1} (dd^c v)^2 \wedge S_j.
$$
Repeating the integration by parts we finally get
\begin{equation} \label{eq:estimateA}
A_m (\delta) \leq 2  m! \, \Vert \varphi \Vert_{\Omega}^{m -1} \kappa_\varphi  (\delta)  \int_{\Omega} (dd^c v)^m \wedge \beta^{n - m} \leq C_1 \kappa_\varphi  (\delta),
\end{equation}
where $C_1 :=  2 m! \, R \, \Vert \varphi \Vert_{\Omega}^{m -1} $.

 To estimate the second term, we need to establish the following estimate $0 < \delta < \delta_0$,

\begin{equation} \label{eq:Festimate}
dd^c \varphi_\delta \leq M_2 \frac{\kappa_\varphi (\delta)}{\delta^2}  \beta, \, \, \text{pointwise on} \, \, \,  \Omega,
\end{equation}
where $M_2 > 0$ is a uniform constant.

Indeed, by differentiating the integral formula $
\varphi_\delta (z) = \varphi \star \chi_\delta (z)$
and by making an obvious change of variables, we obtain for $j, k = 1, \cdots, n$
\begin{eqnarray*}
\partial_j \partial_{\bar{k}} \varphi_{\delta}  (z)
&=&  \delta^{-2} \int_{\C^n}  \varphi (z - \delta \eta)  \partial_j \partial_{\bar{k}} \chi   (\eta) d \lambda_{2 n} (\eta) \\
&=& \delta^{-2} \int_{\C^n}  [\varphi (z - \delta \eta)  - \varphi (z)] \partial_j \partial_{\bar{k}} \chi   (\eta) d \lambda_{2 n} (\eta),
\end{eqnarray*}
where the last equation follows from the fact that by Stokes formula $\int_{\C^n} \partial_j \partial_{\bar{k}} \chi   (\eta) d \lambda_{2 n} (\eta) = 0$ since $\chi$ is a smooth test function with compact support.
Thus the estimate (\ref{eq:Festimate}) follows from the last equation since the support of $\chi$ is contained in the unit ball.

 Now from the inequalities (\ref{eq:projection}) and (\ref{eq:Festimate}), it follows that
 $$
 (dd^c \psi_\delta)^m \wedge \beta^{n-m} \leq  M_2^m \frac{\kappa_\varphi^m (\delta)}{\delta^{2 m}} \beta^n,  \, \, \,  \text{weakly  on} \, \, \,  \Omega.
 $$
 Therefore we have
\begin{equation} \label{eq:estimateB}
B_m (\delta) \leq C_2 \frac{\kappa_\varphi^m (\delta)}{\delta^{2 m}} \int_{\Omega} (u - v)^m \beta^n,
\end{equation}
where $C_2 := M_2^m$.

From (\ref{eq:estimateA}) and (\ref{eq:estimateB}) we conclude that
$$
\int_{\Omega} (u - v)^m  (dd^c\varphi)^m\wedge\beta^{n-m} \leq  C_1 \kappa_\varphi  (\delta) +   C_2 \frac{\kappa^m_\varphi (\delta)}{\delta^{2 m}} \int_{\Omega} (u - v)^m \beta^n.
$$

We want to optimize the right hand side by taking $\delta> 0$ so that 
$$
 \delta^{2 m} \kappa^{1-m}_\varphi (\delta) = \int_{\Omega} (u - v)^m \beta^n = \Vert u-v\Vert_m^m,
$$
i.e.  $\delta = \theta_m (\Vert u-v\Vert_m^m)$, where $\theta_m$ is the reciprocal of the function $t \longmapsto t^{2m} \kappa_\varphi^{1-m} (t)$. This is possible if  $\Vert u - v \Vert_{m}^m \leq \theta_m^{-1} (\delta_0)$ so that  $\delta < \delta_0$. Then applying the previous estimate we obtain the  estimate of the Lemma in this case.

Now assume that $\Vert u - v \Vert_{m}^m >  \theta_m^{-1} (\delta_0)$.  By Lemma \ref{lem:Blocki}, we have
$$
\int_{\Omega}(u-v)^m (dd^c\varphi)^m\wedge\beta^{n-m}\leq m! \Vert \varphi \Vert_{\infty}^m \int_\Omega  (dd^cv)^m\wedge\beta^{n-m} \leq  m!  R \Vert \varphi \Vert_{\infty}^m.
$$
We see that we obtain the inequality of the Lemma \ref{eq:MocEst} by increasing the constant $ C$ consequently.
\end{proof}

\begin{corollary} \label{cor:Lmuestimate} Under the same assumptions as the Theorem \ref{thm:ModC}, we have
$$
  \int_\Omega \vert u - v \vert^m \sigma_m (\varphi) \leq  C (m) \,  \kappa_\varphi \circ \theta_m\left(M  \Vert u - v \Vert_1\right),
$$
  where $M := \left[\Vert u - v\Vert_{\infty}\right]^{m - 1}$.
\end{corollary}
\begin{proof} 
Apply Theorem \ref{thm:ModC} 	and observe that 
$$
\Vert u-v\Vert_m^m \leq  \Vert u-v\Vert_{\infty}^{m - 1} \, \,  \Vert u-v\Vert_1.
$$
The required inequality follows immediately, since $\kappa_\varphi $ is non decreasing.
\end{proof}
\subsection{Global approximants to the solution} 

Let $u  \in \mathcal{SH}_m (\Omega) \cap \mathcal C (\bar \Omega)$. We define the volume mean-values of $u$  as follows:
\begin{equation} \label{eq:volumemeanvalue}
\widehat{u}_\delta (z):= \frac{1}{\tau_{2n}\delta^{2n}} \int_{|\zeta -z|\leq \delta} u(\zeta) dV_{2n}(\zeta), z \in \Omega_{\delta},
\end{equation}
where $\tau_{2n}$ is the volume of the unit ball in $\C^n$.

We need  the following lemma which was proved in \cite{Ch16b} in the H\"older continuous case.  

 \begin{lemma}\label{lem:approximation}
 Let  $u \in \mathcal{SH}_m (\Omega) \cap L^{\infty} (\Omega)$ such that there exists two functions $v , w : \bar{\Omega} \longrightarrow \R$  continuous  in $\bar \Omega$ such that $v \leq u \leq w$ in $ \Omega$ and $v = w$ in $\partial \Omega$.

 Then there exist $\delta_0>0$ small enough, depending on $\Omega$,   such that for any $  0<  \delta < \delta_0$ the function defined by 
 
\begin{equation} \label{eq:approximants}
 	\tilde {u}_{\delta}
 		= \begin{cases}
			\max\{\widehat{u}_{\delta} -  \widehat{\kappa} (\delta) , u\} & \text{ on } \; \Omega_\delta,\\
                   u  & \text{ on } \; \Om\setminus \Omega_\delta,
           \end{cases}
\end{equation}
is  a bounded $m$-subharmonic  function  in $\Omega$ which satisfies the inequalities 
$$
 0 \leq  \tilde {u}_{\delta} - u \leq \widehat{u}_{\delta}  - u \leq \tilde{u}_{\delta}  - u +  \widehat{\kappa} (\delta) , \, \, \, \mathrm{on} \, \, \, \Omega_\delta,
$$
where $\tilde{\kappa} (\delta) := \widehat{\kappa}_v (\delta) + \widehat{\kappa}_w (\delta) + \delta$ for $0 < \delta < \delta_0$.

Moreover $\tilde{u}_{\delta} = u$ in a neighbourhood of $\partial \Omega_\delta$ in $\Omega$.
\end{lemma}
\begin{proof} By the gluing property (see Proposition \ref{prop:basic}),  it is enough to prove that for $0 < \delta < \delta_0$, $ \widehat{u}_{\delta} -   \widehat{\kappa} (\delta) \leq  u $ in $\partial \Omega_\delta$.

 Indeed fix $0 < \delta < \delta_0 < 1$ and  fix $z \in \partial \Omega_\delta$. Then there exists $\zeta \in \partial \Omega$ such that $\vert z - \zeta \vert = \delta$. Hence   
 
 \begin{eqnarray*}
 \widehat{u}_{\delta} (z)  \leq   \widehat{w}_\delta (z) & \leq & w (\zeta) + \widehat{\kappa}_w (\delta)  \\
 & = & v (\zeta)  + \widehat{\kappa}_w (\delta)  \\
 & \leq & v (z) +   \widehat{\kappa}_v (\delta) + \widehat{\kappa}_w (\delta) \\
& <  &  v  (z) +   \tilde{\kappa} (\delta) \leq u(z) +  \tilde{\kappa} (\delta),
 \end{eqnarray*} 
 which proves the required condition. Observe that, since $v$ is continuous, the set   $ \{\widehat{u}_\delta  -  \tilde{\kappa} (\delta) < v(z) \} $ is a neighbourhood of $\partial \Omega_\delta$. Hence  $\widehat{u}_\delta - \tilde{\kappa} (\delta) \leq u $ is a neighbourhood of $\partial \Omega_\delta$ and then  $\tilde{u}_{\delta} = u$ in a neighbourhood of $\partial \Omega_\delta$. 
 \end{proof}
 
 The following estimate will play a crucial role in the proof of Theorem 3 and Theorem 4.
 
\begin{corollary}\label{cor:approximation} Let $\Omega$ be a bounded strictly $m$-pseudoconvex domain and  $\mu$ a positive Borel measure on $\Omega$ with finite mass. Assume there exists  $\varphi \in  \mathcal{SH}_m (\Omega) \cap \mathcal C (\bar \Omega)$ such that $\varphi = 0$ on $\partial \Omega$ and $\mu \leq \sigma_m (\varphi)$ weakly on $\Omega$.
Let $ g \in \mathcal C (\partial \Omega)$ and $u \in \mathcal{SH}_m (\Omega) \cap L^{\infty} ( \Omega)$ satisfying $\sigma_m (u) \leq \mu$ weakly on $\Omega$ and $u = g$ on $\partial \Omega$.

Then there exists two  continuous  functions $v$ and $ w$ in $\bar \Omega$ satisfying the requirements of Lemma \ref{lem:approximation} so that the corresponding functions  $ (\tilde{u}_\delta)_{0 < \delta < \delta_0} $  defined by the formula (\ref{eq:approximants}) satisfy the following estimates:

$$
\int_{\Omega_\delta} (\tilde{u}_\delta - u)^m d \mu  \leq C (m,\mu) \,  \kappa \circ \theta_m( D \delta), \, \, 0 < \delta < \delta_0,
$$
where $C(m,\mu) = C (m,\Omega,g, \mu) >0$ and $D = D (m,n,\varphi,g) > 0$ are uniform constants, $\kappa(\delta) := \kappa_\varphi(\delta) + \kappa_g(\sqrt{\delta})$.  
\end{corollary}

\begin{proof}
We want to apply Lemma \ref{lem:approximation} and  Corollary  \ref{cor:Lmuestimate}. To this end, we need to construct two functions $v$ and $w$ satifying the requirement of the Lemma \ref{lem:approximation}.  Let $w$ be the maximal $m$-subharmonic function on $\Omega$ with boundary values $g$. By \cite{Ch16b}, we have $\kappa_w  (\delta) \leq \kappa_g (\sqrt{\delta})$ and by the comparison principle we have $u \leq w$ on $\Omega$.   

 Moreover,  the function $v := \varphi + w$ is $m$-subharmonic on $\Omega$, continuous on $\bar{\Omega}$ with $\kappa_v (\delta)  \leq \kappa_\varphi (\delta)  + \kappa_g (\sqrt{\delta})$ and $v = g$ on $\partial \Omega$. Since   $ \sigma_m (v) \geq \sigma_m (\varphi) $ and $\sigma_m (u) \leq \mu \leq \sigma_m (\varphi$ weakly on $\Omega$, it follows from the comparison principle that
$v \leq u$ on $\Omega$.

 Therefore we can apply Lemma \ref{lem:approximation} to construct  global approximants $ (\tilde{u}_\delta)_{0 < \delta < \delta_0} $  given by the formula (\ref{eq:approximants}). Since  $ \tilde{u}_\delta  \geq  u$ in $\Omega_\delta$, and $\tilde{u}_\delta = u$ in a neighbourhood of $\Omega \setminus \Omega_\delta$,  it follows from Lemma \ref{lem:Comparison Principle} that

  $$
 \int_{\Omega} (dd^c \tilde{u}_\delta)^m \wedge \beta^{n-m} \leq  \int_{\Omega} (dd^c u)^m \wedge \beta^{n-m} \leq \mu (\Omega).
 $$

  By Corollary  \ref{cor:Lmuestimate}, we have for $0 < \delta < \delta_0$
 \begin{equation} \label{eq:FINEQ1}
 \int_{\Omega } (\tilde{u}_{\delta} - u)^m d\mu \leq C_m \,  \kappa_{\varphi}  \circ \theta_m  (M \Vert \widetilde u_{\delta} - u\Vert_1),
 \end{equation}
 where $M := (\text{osc}_{\Omega} u)^{m -1}$ and  $C_m = C (m,\Omega,g, \mu) >0$ is a uniform constant.
 
 Now observe that $\tilde{u}_{\delta} - u = 0$ on $\Omega \setminus \Omega_\delta$ and  $\tilde{u}_{\delta} - u \leq \widehat{u}_\delta - u$ on $\Omega_\delta$. This yields 
  \begin{equation} \label{eq:FINEQ2}
 \Vert \tilde{u}_{\delta} - u \Vert_1 \leq \int_{\Omega_\delta} (\widehat{u}_\delta - u) d \lambda_{2n}.
  \end{equation}

By Lemma \ref{lem:Poisson-Jensen} , we have 
$$\int_{\Omega_\delta} (\widehat{u}_\delta - u ) d \lambda_{2 n} \leq b_n \delta \Vert u\Vert_{L^1(\Omega)} \leq b'_n \delta,
$$
 where $b'_n > 0$ is a positive uniform constant depending on uniform bounds on $v$ and $w$. 
 
 Hence from (\ref{eq:FINEQ2})   we conclude that 
 
 \begin{equation} \label{eq:FINEQ3}
\Vert \tilde{u}_{\delta} - u \Vert_1  \leq D' \, \delta,
 \end{equation}
 where $D' = D' (m,n,\mu) > 0$ is a uniform constant.
 
  Moreover since $\varphi + w \leq u \leq w$, we have 
   \begin{equation} \label{eq:FINEQ4}
  M  := (\text{osc}_{\Omega} u)^{m -1} \leq (\text{osc}_{\Omega} w + \Vert \varphi \Vert_{\bar{\Omega}})^{m -1}.
  \end{equation}
  The conclusion follows from (\ref{eq:FINEQ1}), (\ref{eq:FINEQ3}) and (\ref{eq:FINEQ4}).
 \end{proof}

 When $g \in C^{1,1} (\partial \Omega)$ (e.g. $g \equiv 0$), it is possible to improve the estimate of Corollary \ref{cor:approximation}.
 \begin{corollary} \label{cor:improvement1} Under the same assumptions as Corollary \ref{cor:approximation}, we assume moreover that $g \in C^{1,1} (\partial \Omega)$. Then we have for $0< \delta < \delta_0$,
$$
\int_{\Omega_\delta} (\tilde{u}_\delta - u)^m d \mu  \leq C (m,\mu) \,  \kappa \circ \theta_m( D \delta^2).
$$
where $C(m,\mu) = C (m,\Omega,g, \mu) >0$ and $D = D (m,n,\varphi,g) > 0$ are uniform constants, $\kappa(\delta) := \kappa_\varphi(\delta) $.  
 \end{corollary}
 \begin{proof}
The proof is the same as the previous one if we can improve the estimate  (\ref{eq:FINEQ3}) by showing that there exists a constant $D'' > 0$ depending only on the data $\mu$ and $g$ such that

$$
\int_{\Omega_\delta} (\widehat{u}_\delta - u ) d \lambda_{2 n} \leq D" \delta^2.
$$

Indeed by Lemma \ref{lem:Poisson-Jensen} , we have 
$$\int_{\Omega_\delta} (\widehat{u}_\delta - u ) d \lambda_{2 n} \leq b_n \delta^2 \Vert \Delta u\Vert_\Omega.,
$$
 Therefore it's enough to show the uniform bound
 $$
\int_\Omega dd^c u \wedge  \beta^{n-1} \leq D".
$$
where $D" = D" (n,m,\mu,g) >$ is a uniform constant.

To prove this estimate we argue as follows (see \cite{GKZ08}, \cite{Ch16a}, \cite{BZ20}).

 Assume first that $g := 0$ and let $u_0 := U_{0,\mu}$ be  the solution to the Dirichlet problem (\ref{eq:DirPb}) with $0$ boundary values and $\mu$ as the right had side. Then by Lemma \ref{lem:Cegrell}  we have
$$
\int_\Omega dd^c u_0 \wedge \beta^{n-1} \leq a_{n,m} \, \mu (\Omega)^{1 \slash m}.
$$

When $g \in C^{1,1} (\partial \Omega)$ it can be extended as a $C^{1,1}$ function $\tilde g$ on a neighborhood of $\bar{\Omega}$. Then for a large canstant $A > 0$ the function $ A \rho + \tilde g$ is $m$-subharmonic on $\Omega$ and then the function $v := u_0+ A \rho + \tilde g $ is  $m$-subharmonic on $\Omega$. Moreover $v = g$ on $\partial \Omega$ and $dd^c v)^n \geq \mu$ on $\Omega$.
By the comparison principle we then have $v \leq u$ on $\Omega$ and by Lemma \ref{lem:Comparison Principle} we obtain
\begin{eqnarray}
\int_\Omega dd^c u \wedge \beta^{n-1}  &\leq & \int_\Omega dd^c v \wedge \beta^{n-1}  \leq D",
\end{eqnarray}
where
$$
D" := a_{n,m} \mu (\Omega)^{1 \slash m}  +  A \int_\Omega dd^c \rho \wedge \beta^{n-1} + \int_\Omega dd^c \tilde g \wedge \beta^{n-1} 
$$

\end{proof}

\section{Continuity of the pluripotential of a diffuse measure}

\subsection{Diffuse Borel measures}

We will use the following terminology from Potential Theory (see \cite{Po16}). 
\begin{definition}  \label{def:cap-domination} Let $\mu$ be  a positive Borel measure on $\Omega$. 
 
 1. We say that $\mu$ is diffuse  with respect to  the capacity $  \text{c}_m = \text{c}_m (\cdot,\Omega)$ if  $\mu (K) = 0$ whenever $K \subset \Omega$ is a compact set with $ \text{c}_m (K,\Omega) = 0$. 
 
2. We associate to $\mu$ its "modulus of diffusion" wrt the $m$-Hessian capacity defined as follows :
\begin{equation} \label{eq:AC}
 \Gamma _\mu (t) = \Gamma _{\mu,m} (t)  := \sup \{ \mu (K) ; K \subset \Omega \, \, \, \hbox{is compact },  \, \, \,   \text{c}_m (K,\Omega) \leq t\}.
\end{equation}
It follows from the definition that $\Gamma_\mu$ is non decreasing right continuous function on $\R^+$ which satisfies the following property: for any compact set $K \subset \Omega$, we have
\begin{equation} \label{eq:capdomination}
\mu (K) \leq \Gamma _\mu \left(\text{c}_m (K)\right),
\end{equation}
where $\text{c}_m (K) = \text{c}_m (K,\Omega)$. 

 Observe that by outer regularity of the measure $\mu$ and the capacity $c_m$, this inequality is satisfied for any Borel set $K \subset \Omega$.
 
3. If $\Gamma$ is a non-decreasing right continuous function on $\R^+$, we say that $\mu$ is $\Gamma$-diffuse (with respect to  the $m$-Hessian capacity)  if for any compact subset $K \subset \Omega$, with $\text{c}_m (K) \leq 1$, 
\begin{equation} \label{eq:capdomination}
 \mu (K) \leq  \Gamma \left( \text{c}_m (K)\right). 
\end{equation}
This means that $\Gamma_\mu (t) \leq \Gamma (t),$ for any $t \in [0,1]$.
\end{definition}

 Let us mention that S.  Ko\l odziej was the first to relate the domination  of the measure $\mu$ by the Monge-Amp\`ere capacity to the regularity of the solution to complex Monge-Amp\`ere equations (see \cite{Kol96}). 
 
The following lemma is easy to prove (see \cite{Po16}).
\begin{lemma}
A   positive Borel measure $\mu$ on $\Omega$ is  diffuse  (with respect to  the $m$-Hessian capacity) if and only if $\lim_{t \to 0^+} \Gamma_\mu (t) = 0$.
\end{lemma}
 
 Let us give a simple example.
 \begin{example} Let $\phi \in  \mathcal{SH}_m(\Omega) \cap L^{\infty} (\Omega)$ and  $\sigma_\phi  := (dd^c \phi)^m \wedge \beta^{n-m}$ be its $m$-Hessian measure. Set $M := \hbox{osc}_{\Omega} \phi$. 
 Then from the definition of the $m$-Hessian capacity, we have for any compact subset $K \subset \Omega$,
 $$
 \sigma_\phi   (K) \leq A  \, \text{c}_m (K), \, \, \, \hbox{where} \, \, \, A := M^m.
  $$
  This implies that the measure $\sigma_\phi$ is diffuse  with respect to the $m$-Hessian capacity on $\Omega$ and  $\Gamma_{\sigma_\phi}  (t) \leq A t$ for any $t \in \R^+$.  
  
 An example of Ko\l odziej shows that there exits a Borel measure $\mu$ such that $\mu \leq \text{c}_n $, but $\mu$ is not the Monge-Amp\`ere of a bounded plurisubharmonic function (see \cite{Kol96}).
 \end{example} 

The following examples due to Dinew and Ko\l odziej   are more involved.
\begin{example}
1.Assume that $1 \leq m < n$. Then  Dinew and Ko\l odziej  proved in \cite{DK14} that the volume measure $\lambda_{2 n}$ is diffuse with respect to the $m$-Hessian capacity. Namely for any $1 < r < \frac{n}{n - m}$, there exists a constant $N (r) > 0$ such that for any compact subset $K \subset \Omega$, 
\begin{equation} \label{eq:DK}
\lambda_{2 n} (K) \leq N (r) \text{c}_m (K)^{r}.
\end{equation}
 Observe that this estimate is sharp in terms of the exponent when $m < n$. This can be seen by taking $\Omega = \B$ the unit ball and  $K := \bar{\B_r} \subset \B$ the closed ball of radius $r \in ]0,1[$, since $\text{c}_m(\bar{\B}_r ,\B) \approx  r^{2 (n-m)}$ as $r \to 0$ (see \cite{Lu12}).

  Let $ 0 \leq f \in L^p (\Omega)$ with $p > n \slash m$. Then $ \frac{n (p-1)}{p (n - m)} > 1$. By H\"older inequality and inequality (\ref{eq:DK}) we obtain: for any  $1 < \tau < \frac{n (p-1)}{p (n - m)}$ there exists a constant $M (\tau) > 0$ such that for any compact set $K \subset \Omega$, 
\begin{equation*} 
\int_K f d \lambda_{2 n} \leq  M (\tau) \Vert f\Vert_p  \text{c}_m (K)^{\tau}.
\end{equation*}

2. When $m=n$ the domination is much more precise. It was proved in  \cite{ACKPZ09} that for any $0< b < 2 n $, there exists a constant $B > 0$ such that for any compact subset $K \subset \Omega$, 
 
 \begin{equation} \label{eq:ACKPZ}
 \lambda_{2 n} (K) \leq  \, B \,  \text{c}_n (K) \,  \exp \left( -b \left[\text{c}_n (K)\right]^{- 1 \slash n}\right).
 \end{equation}

Let  $ 0 \leq f \in L^p (\Omega)$ with $p > 1$, then by H\"older inequality and inequality (\ref{eq:ACKPZ}),for any $0< b < 2 n (p-1) \slash p$, there exists a constant $B' > 0$ such that for any compact set $K \subset \Omega$, 
\begin{equation*} 
\int_K f d \lambda_{2 n} \leq  B' \Vert f\Vert_p  \,  \text{c}_n (K) \exp \left( -b \left[\text{c}_n (K)\right]^{- 1 \slash n}\right).
\end{equation*}
\end{example}

Theorem 2 will provide us with new examples.
 
 The condition (\ref{eq:capdomination}) plays an important role in the following stability result which will be a crucial point in the proof of our theorems (see \cite{EGZ09, GKZ08, Ch16a}).

\subsection{Uniform a priori estimates}
The following lemma is elementary, but it turns out to play a crucial role.
\begin{lemma}\label{lem:Kolo}
Let $f :\mathbb{R^+}\rightarrow\mathbb{R^+}$ be a decreasing right continuous function  such that $\lim_{s \to + \infty} f (s) = 0$ and let $\eta :\mathbb{R^+}\rightarrow\mathbb{R^+}$ be a  non-decreasing function which satisfies the following Dini condition
\begin{equation} \label{eq:DC0}
\int_{0^+} \frac{\eta (t)}{t} d t< + \infty.
\end{equation}
Assume that for any $t \in [0,1]$ and any $s > 0$, we have
\begin{equation} \label{eq:Etadecay}
t \, f (s+t) \, \leq  \,  f (s) \cdot  \eta ( f (s)).
\end{equation}
Then $f (s)=0$ for all $s\geq S_{\infty}$, where 
$$
S_{\infty}:= s_0 + e \int_{0}^{ e f (s_0)}  \frac{\eta (t)}{t} d t, 
$$
and $s_0 \geq 0$ satisfies the condition
$$
\eta (f (s_0)) \leq 1 \slash e < 1.
$$
\end{lemma}
Observe that the Dini condition implies that $\lim_{t \to 0^+} \eta (t) = 0$. 
This Lemma is a reformulation of a Lemma of Kolodziej \cite{Kol05} in the spirit of  \cite{EGZ09} and \cite{BGZ08}). Its proof is a variant of the proof of \cite[Lemma 2.4]{EGZ09}. We will give it here for the convenience of the reader.
\begin{proof} 
Since $\lim_{s \to + \infty} f(s) = 0$, there exits $s_0 > 0$ such that $f(s_0) < 1 \slash e$. If $f (s_0) = 0$ we are done. If $f (s_0) > 0$, there exists $s > s_0$ such $f (s)) <   f (s_0) \slash e$ since $\lim_{s \to  + \infty} f (s)  = 0$. Therefore we can set  
$$
s_1 :=  \inf \{ s > s_0  \, \, ; \, \,  f (s) <  f (s_0) \slash e \}.
$$ 
By (\ref{eq:Etadecay}) we have
$$
f (s_0 + 1) \leq f (s_0) \eta (f(s_0)) < f (s_0) \slash e,
$$
which implies that $s_0 < s_1 \leq s_0 + 1$.

By definition of $s_1$,  there exists a sequence $s'_k$ decreasing to $s_1$ such that $ f (s'_k) <  f (s_0) \slash e$ for any $k > 0$. Since $f$ is right continuous, it follows that $f (s_1) = \lim_{k \to + \infty} f (s'_k) \leq f (s_0) \slash e$.

Thus we have proved that $s_0 < s_1 \leq s_0 + 1$ and $f (s_1) \leq f (s_0) \slash e$.

 We will construct  by induction an increasing  sequence $(s_j)_{j \geq 0}$ of positive numbers such that or any $j \in \N$
 $$
 s_j < s_{j + 1} \leq s_j + 1 \, \, \, \hbox{and} \, \, \, f (s_{j + 1} )\leq f (s_j) \slash e
 $$
Indeed assume by induction that for a fixed $j \geq 1$,   $s_1, \cdots, s_j$ are constructed with the required properties . Then the number 
$$
s_{j + 1} :=  \inf \{ s > s_j \, \, ; \, \,  f (s) <  f (s_j) \slash e \}.
$$
is well defined and by the same reasoning for $s_1$ we see that it satisfies the required properties. 

On the other hand, since $f (s_{j + 1} )\leq f (s_j) \slash e$, from (\ref{eq:Etadecay}), it follows that for any $s \in ]s_j , s_{j +1}[$ we have
$$
 (s - s_j) f (s) \leq f (s_{j}) \eta (f (s_{j})) \leq  e f (s)  \eta (f (s_{j})),
$$
since $ s_j < s < s_{j+1}$ and then $f (s) \geq f (s_{j}) \slash e$.

Therefore for any $j \in \N$ and $s \in ]s_j, s_{j + 1}[$, $ s-s_j \leq e  \eta (f (s_{j}))$, hence for any $j \in \N$,

$$
s_{j + 1} - s_j \leq e  \eta (f (s_{j}))
$$
Moreover since $f (s_{j} )\leq f (s_{j-1}) \slash e$ for any $j\geq 1$, it follows that 
$f (s_{j } )\leq f (s_0) \slash e^j$ and then 
$$
s_{j + 1} - s_j \leq  e {\eta} (f (s_0)  e^{-j}),
$$
for all $j \in \N$ since $\eta$ is non decreasing.

Therefore
\begin{eqnarray*}
s_{\infty} := \lim_{j \to + \infty} s_j = s_0 + \sum_{j \geq 0} (s_{j + 1} - s_j) & \leq & s_0 + e \sum_{j =0}^{+ \infty} \eta (f (s_0)  e^{-j}) \\
&\leq & s_0 +  e \int_0^{+\infty} {\eta} \left(f (s_0)  e^{- x+ 1}\right)  d x.
\end{eqnarray*}
 A simple change of variables yields
 $$
 s_{\infty} \leq s_0 + e \int_0^{e f (s_0)} \frac{\eta (t)}{t} d t.
 $$ 
 Recall that by construction  $s_j \leq s_{\infty}$ and $f (s_{j + 1} )\leq f (s_0) \slash e^j$  for any $j \in \N$. Therefore for any $j \in \N$

$$ 
0 \leq f (s_{\infty})  \leq f (s_{j + 1} )\leq f (s_0) \slash e^j,
$$ 
which implies that  that $f (s_{\infty}) = 0$. 

Therefore if we define
$$
S_{\infty} := s_0 + e \int_0^{e f (s_0)}  \frac{\eta (t)}{t} d t,
$$
 we conclude that $f (s) = 0$ for any $s \geq S_{\infty}$.
 
 Finally observe that, since  $\eta (f (s_0)) \leq 1\slash e < 1$, we have $f (s_0) \leq \eta^{-1} (1\slash e)$.
 \end{proof}

\begin{remark} 1. Observe that if $ \eta (f (0)) \leq 1 \slash e$ then $ f (s) = 0$ for any $s \geq  S_{\infty}$, where
$$
S_{\infty} := e \int_0^{e f (0)}  \frac{\eta (t)}{t} d t.
$$

2. If $\eta$ does not have the monotonicity property, we can replace in the statement of the lemma $\eta$ by the least non decreasing majorant function of $\eta$ define by
$$
\bar{\eta} (t) := \sup \{\eta (s) \geq  0 ; s \leq t\} , \, t \geq 0.
$$
\end{remark}

\smallskip
We now deduce a uniform a priori estimate on solutions to  complex Hessian equations. 

\begin{lemma} \label{lem:CapEst} Let $u, v \in  \mathcal{SH}_m (\Omega) \cap L^{\infty} (\Omega)$ be such that 
$$
 \liminf_{z\to \partial\Omega}(u-v)(z)\geq 0.
$$
 Then for any $t > 0, s > 0$, we have

 \begin{equation}\label{eq:CapEst}
 t^m {Cap}_m(\{u< v -s-t\},\Omega)\leq \int_{\{u<v -s\}}(dd^cu)^m\wedge\beta^{n-m}.
\end{equation}
\end{lemma}
The  lemma is well known. It follows from the comparison principle (see \cite{Kol96}, \cite{EGZ09}, \cite{GKZ08}, \cite{Ch16a}, \cite{KN20a}).

\begin{corollary} \label{cor:UnifEst} Let  $\mu$ be a positive Borel measure on $\Omega$ with finite mass.
Assume that $\mu$ is $\Gamma$-diffuse  with respect to  the $m$-Hessian capacity and the function  $ \gamma (t) := \Gamma (t) \slash t$ is non-decreasing and
satisfies the following Dini type condition
\begin{equation} \label{eq:DiniConditionMu}
\int_{0^+}  \frac{\gamma^{1 \slash m} (t)}{t} d t < + \infty.
\end{equation}
Assume that $u, v  \in \mathcal{SH}_m (\Omega) \cap L^{\infty} (\Omega)$ satisfy   $\liminf_{z \to \partial \Omega} (u - v) \geq 0$ and $(dd^c u) ^m\wedge \beta^{n-m}  \leq \mu$ on $\Omega$.  Then we have the following uniform estimate on $\Omega$ 

\begin{equation} \label{eq:UnifEst} 
 \sup_{\Omega} (v - u) \leq   2  e \left(\mu (\Omega)\slash a\right)^{1 \slash m}  + e \int_0^{a} \frac{\gamma^{1 \slash m} (t)}{t} d t,
 \end{equation}
  where
  $$
   a := e^m \gamma^{-1} (1\slash e^m).
  $$
  
In particular  if $g \in C^0(\partial \Omega)$ and $\lim_{z \to \zeta} u (z) = g (\zeta)$ for any $\zeta \in \partial \Omega$, then 
 \begin{equation} \label{eq:UnifOsc} 
 \mathrm{osc}_{\bar \Omega} u \leq \mathrm{osc}_{\partial  \Omega} g +2  e \left(\mu (\Omega)\slash a\right)^{1 \slash m}  + e \int_0^{a }\frac{\gamma^{1 \slash m} (t)}{t} d t.
 \end{equation}
\end{corollary}
\begin{proof} Set   $\liminf_{z \to \partial \Omega} (u - v) \geq 0$.
Then we can apply Lemma \ref{lem:CapEst} and get the following estimate for any $s, t > 0$
\begin{eqnarray} \label{eq:CapEst2}
t^m \hbox{c}_m  (\{u - v <  - s - t\},\Omega)  &\leq & \int_{\{u - v < - s\}} (dd^c u)^m \wedge \beta^{n-m}  \nonumber \\
&  \leq & \mu (\{u - v < - s\} ).
\end{eqnarray}

 By definition of $\Gamma$, we deduce that for any $s, t > 0$
\begin{eqnarray*}
t^m\hbox{c}_m  (\{u - v <  - s - t\},\Omega)\, \, \, \leq  \, \, \, \Gamma  (\hbox{c}_m  (\{u - v <  - s \},\Omega)). 
\end{eqnarray*}
 Define $f (s) :=\hbox{c}_m  (\{u - v <  - s \},\Omega)^{1 \slash m}$ for $s > 0$. Then we see that the condition of the Lemma \ref{lem:Kolo} is satified with  $\eta (t) := \gamma^{1 \slash m} (t^m)$.

 Applying  Lemma \ref{lem:Kolo} we conclude that
 $ f (s) = 0 $ for $s \geq S_{\infty}$. This means that $u \geq  v - S_{\infty}$ outside a set of zero capacity. Since such set is of Lebesgue measure zero, it follows that $u \geq  v - S_{\infty}$,
  where
  $$
   S_{\infty} \leq  s_0 + \int_0^{e \eta^{-1} (1\slash e)}   \frac{\eta (t)}{t} d t  =    s_0 + (1 \slash m) \int_0^{a} \frac{\gamma^{1 \slash m} (t)}{t} d t,
 $$
 and
$$
 a :=  \left[e \eta^{-1} (1\slash e)\right]^m = e^m \gamma^{-1} (1\slash e^m).
 $$ 
 We need to establish a uniform estimate on the initial time $s_0$. 
 Recall that $s_0$ satisfies $\eta (f (s_0)) \leq 1 \slash e.$ This condition is equivalent to the following one   
 $$
\hbox{c}_m  (\{u <  v   - s_0\},\Omega)  \leq a e^{- m}. 
 $$
 Observe that by (\ref{eq:CapEst}) we have for $t > 0$
 $$
 \hbox{c}_m  (\{u -  v <  - 2 t\})  \leq \mu (\Omega) \slash t^m
 $$
 Then choosing $t_0 =  e \left(\mu (\Omega)\slash a\right)^{1 \slash m}$ and  $s_0 := 2 t_0 $, we obtain the estimate  $f (s_0) \leq a$ and then 
 $$
 S_{\infty} \leq   2  e \left(\mu (\Omega)\slash a\right)^{1 \slash m}  + \int_0^{a} \frac{\gamma^{1 \slash m} (t)}{t} d t.
  $$
 Therefore from this upper bound and (\ref{eq:Ubound}),  we obtain the uniform estimate (\ref{eq:UnifEst}).

 On the other hand by the classical maximum principle we also have $u \leq M = M_{g} := \max_{\partial \Omega} g$ in $\Omega$.
 
 Therefore applying the previous estimate with $v := \min_{\partial \Omega } g$, we obtain the following uniform bound on $u$ in $\Omega$
 \begin{equation} \label{eq:Ubound}
  \min_{\partial \Omega} g - S_{\infty} \leq u \leq \max_{\partial \Omega} g,
 \end{equation}
 which proves the estimate (\ref{eq:UnifOsc}).
 \end{proof}

\subsection{Existence of a continuous solution : Proof of Theorem 1}

We first prove a weak stability theorem in terms of capacity in the spirit of  a similar result of \cite{EGZ09}.

 \begin{lemma}\label{lem:CapacityStability}  
Let  $\mu$ be a positive Borel measure with finite mass on $\Omega$. Assume that $\mu$ is $\Gamma$-diffuse with respect to the $m$-Hessian capacity and the function  $ \gamma (t) := \Gamma (t) \slash t$ is non-decreasing and satisfies the Dini type condition (\ref{eq:DiniConditionMu}).  

Then  there exists a  uniform constant $B = B (m,\gamma,\mu(\Omega)) > 0$  such that for any $\varepsilon > 0$, and any $u, v \in \mathcal{SH}_m(\Omega)$ such that 
$$
 \liminf_{z\to\partial\Omega}(u-v)(z)\geq 0 \, \, \, { and} \, \, \, \, (dd^c u)^m \wedge \beta^{n-m} \leq \mu,
$$
 in the sense of currents on $\Omega$,  we have
$$
\sup_{\Omega}(v-u)_+ \leq\varepsilon+ B \int_0^{ \varsigma (\varepsilon)} \frac{\gamma^{1 \slash m}  (t)}{t} d t,
$$
where
$$
\varsigma (\varepsilon) := e^m  \left[\mathrm{c} _{m} (\{v - u>\varepsilon\}, \Omega)\right].
$$
\end{lemma}

\begin{proof} 
 Fix  $\varepsilon > 0$ and apply Lemma \ref{lem:CapEst} with $t \in [0,1] $ and $s + \varepsilon$. Then  we obtain 
\begin{equation}\label{equa}
t^m \mathrm{c} _{m} (\{u-v<-\varepsilon-t-s\},\Omega)\leq \mu \left(\{u-v<-\varepsilon-s\}\right).
\end{equation}
 Set $ f (s) := \left[\mathrm{c} _{m} (\{u-v<-\varepsilon-s\},\Omega) )\right]^{1 \slash m}$ for $s \in \R^+$.
 Then by the domination condition  we deduce that for $t \in [0,1]$ and $s \in \R^+$
$$
t f (s + t) \leq f (s)  \gamma^{1 \slash m} (f (s)^m)
.$$
Now we can apply Lemma \ref{lem:Kolo} with $\eta (t) :=  \gamma^{1 \slash m} (t^m)$. 

There are two cases to be considered:

1) If $\eta (f (0)) \leq 1 \slash e$, then  by   Lemma \ref{lem:Kolo} we conclude that $ \mathrm{c} _{m} (\{u-v<-\varepsilon-s\},\Omega) = 0$ if $s \geq S_{\infty}$ i.e. 
$$
\sup_{\Omega} (v-u) \leq  \varepsilon + S_{\infty} = \varepsilon + \int_0^{e f (0)} \frac{\eta (t)}{t} d t.
$$
A simple change of variables leads to the estimates.
\begin{equation} \label{eq:estimation1}
\sup_{\Omega} (v-u) \leq  \varepsilon + (1 \slash m) \int_0^{\varsigma  (\varepsilon)} \frac{\gamma^{1 \slash m} (t)}{t} d t,
\end{equation}
where $\varsigma  (\varepsilon) = e^m f(0)^m = e^m  \left[\mathrm{c} _{m} (\{v - u>\varepsilon\}, \Omega)\right]$.

2) If  $\eta (f (0)) > 1 \slash e$, then 
$$ \varsigma  (\varepsilon) := e^m  \left[\mathrm{c} _{m} (\{v - u>\varepsilon\}, \Omega)\right] = e^m f (0)^m  \geq   e^m \left[\eta^{-1} (1 \slash e)\right]^m =: a.
$$
  Hence
$$
\int_0^{\varsigma  (\varepsilon)} \frac{\gamma^{1 \slash m} (t)}{t} d t \geq A := \int_0^{ a} \frac{\gamma^{1 \slash m} (t)}{t} d t.
$$

On the other hand since $u \geq v$ in $\partial \Omega$, by the uniform estimate, we have 
$v - u \leq \max_{\partial \Omega} u - u \leq \mathrm{osc}_{\Omega} u \leq M  = M (\gamma, \mu(\Omega)),$ hence if we let $B_0:= M  A^{-1}$ we get

\begin{equation} \label{eq:estimation2}
\sup_{\Omega} (v-u) \leq  B_0 \int_0^{\varsigma  (\varepsilon)} \frac{\gamma^{1 \slash m} (t)}{t} d t 
\end{equation}
Comparing the estimates (\ref{eq:estimation1}) and (\ref{eq:estimation2}) we obtain the estimate required in the theorem with  
$B := \max \{B_0,1 \slash m\}$.
\end{proof}

We prove Theorem 1 on the existence of a continuous solution to the Dirichlet problem for diffuse measures satisfying the Dini condition (\ref{eq:DiniConditionMu}).

\smallskip
\smallskip

{\it Proof of Theorem 1}. The proof will be done in three steps.

1) {\it Existence of a bounded solution.} Indeed, since $\mu$ is diffuse with respect to  $c_m$, it follows from the generalized Radon-Nikodym theorem that there exits a function $v \in \mathcal{SH}_m (\Omega) \cap L^{\infty} (\Omega)$, $v < 0$ and $F \in L^1 (\Omega,\sigma_m(v))$ such that $\mu = F \sigma_m (v)$ on $\Omega$ ( see  \cite{Ceg98}, \cite{Lu12}).

  Let $(K_j)$ be an exhaustive sequence of compact sets in $\Omega$. Set for each $j \in \N$, $F_j := {\bf 1}_{K_j} \min \{F,j^m\}$  for $j \in \N$. Then $\mu_j := F_j \sigma_m (v)$ is a Borel measure with compact support such that $\mu_j  \leq j^m \sigma_m (v) = \sigma_m (v_j)$ with $ v_j := j v  \in  \mathcal{SH}_m (\Omega) \cap L^{\infty} (\Omega)$ and $v_j < 0$. We can always modify $v_j$ near the boundary to construct a new function $\tilde v_j$ such that $\tilde v_j = v_j$ in a neighborhood of $K_j$ and $ \tilde v_j = 0$ near the boundary $\partial \Omega$.
  
  Let $U_g$ be the maximal $m$-subharmonic function in $\Omega$, continuous in $\bar \Omega$  such that $U_g = g $ in $\partial \Omega$. Then the function $w_j := \tilde v_j + U_g  \in \mathcal{SH}_m (\Omega) \cap L^{\infty} (\Omega)$ satisfies $ \mu_j \leq \sigma_m (w_j)$ on $\Omega$ and $w_j = g$ on $\partial \Omega$  in the sense that $\lim_{z \to \zeta} w_j(z) = g(\zeta)$ for any $\zeta \in \partial \Omega$. By the bounded subsolution theorem \cite[Theorem 2.2]{N13}, there exists $u_j \in \mathcal{SH}_m (\Omega) \cap L^{\infty} (\Omega)$ such that  $\sigma_m (u_j) = \mu_j = F_j \sigma_m (v)$ and  $\lim_{z \to \zeta} u_j(z) = g(\zeta)$ for any $\zeta \in \partial \Omega$.
By the comparison principle, since $(\mu_j)$ is increasing, the sequence $(u_j)_{j \in \N}$ is decreasing and by Corollary \ref{cor:UnifEst}, the sequence $(u_j)$ is uniformly bounded on $\Omega$. Therefore it converges to $u \in  \mathcal{SH}_m (\Omega) \cap L^{\infty} (\Omega)$. By the continuity of the Hessian operator  with respect to  decreasing sequences, it follows that $\sigma_m (u) = \mu$ weakly on $\Omega$.

\smallskip

2) {\it Boundary values of the solution.} Since $\sigma_m (u_j) \leq \sigma_m (u_k) \leq \mu$ for $k \geq j$, it follows that the measures $\mu_j$ are uniformly $\Gamma$-diffuse  with respect to  $c_m$ and then by  Lemma \ref{lem:CapacityStability}, there is a uniform constant $B$ such that for any $ j$, 
\begin{equation} \label{eq:UnifEst1}
\sup_{\Omega}(u_j-u) \leq\varepsilon+ B \int_0^{ \varsigma_{j} (\varepsilon)} \frac{\gamma^{1 \slash m}  (t)}{t} d t,
\end{equation}
where
$$
\varsigma_{j} (\varepsilon) := e^m  \, \mathrm{c} _{m} (\{u_j - u>\varepsilon\}).
$$

Observe that since $\lim_{j \to + \infty} \varsigma_{j} (\varepsilon) = 0$. Indeed  applying  Lemma \ref{lem:CapEst} with $s = t = \varepsilon \slash 2$, we have for any $j \in \N$

 \begin{eqnarray*} \label{eq:CapEst2}
  {Cap}_m(\{u_j - u >  \varepsilon \},\Omega) & \leq  & t^{-m}  \int_{\{u_j - u > t \}}(dd^cu)^m\wedge\beta^{n-m} \\
 & \leq & t^{- m - 1} \int_{\Omega} (u_j - u) (dd^c u)^m\wedge\beta^{n-m},
\end{eqnarray*}
and the right hand side converges to $0$ by the monotone convergence theorem. Therefore by (\ref{eq:UnifEst1}), it follows that   the sequence $(u_j)$ converges uniformly to $u$ in $\Omega$. Hence $u = g$ in $\partial \Omega$ in the sense that  $\lim_{z \to \zeta}  u (z) = g (\zeta)$ for any $\zeta  \in \partial \Omega$.

\smallskip

3) {\it Continuity of the solution $u$.} By Lemma \ref{lem:appximationwithbdv}, there exists  a decreasing sequence  $(w_j)$ of continuous $m$-subharmonic functions in $\bar{\Omega}$ which converges to $u$ pointwise in $\Omega$ and such that $w_j = g$ on $\partial \Omega$. 

Fix $\e > 0$.  Since   $\lim_{z \to \zeta} (u (z)- w_j (z)) =  0$ for any $\zeta \in \partial \Omega$ and  $j \in \N$, applying Lemma \ref{lem:CapacityStability} with $u $ and $v = w_j $, we deduce   that for any $j \in \N$
$$
 \sup_{\Omega} (w_j - u) \leq \e + B \int_0^{\varsigma _j (\e)} \frac{\gamma^{1 \slash m} (t)}{t} d t,
$$

where
$$
\varsigma _j (\e) = e^m \, \mathrm{c}_m (\{ w_j > u + \e\}).
$$
Since $(w_j)$ decreases to $u$ pointwise in $\Omega$, it follows as before that it converges in capacity over any compact subset of $\Omega$. Observe that for any $j \in \N$, we have $\{u < w_j - \e\} \subset \{u < w_0 - \e\}  \Subset \Omega$. Hence $\lim_{j \to + \infty} \varsigma_j (\e) = 0$, and then
$$
\lim_{j \to + \infty} \max_{\bar \Omega}(w_j - u) \leq \e.
$$
As $\e > 0$ is arbitrary, it follows that the sequence $(w_j)$ converges uniformly in $\bar{\Omega}$ to $u$, hence $u$ is continuous in $\bar \Omega$.  This finishes the proof of Theorem 1.
\smallskip
\smallskip

In the case of infinite mass, we can prove the following result using Theorem 1.
\begin{corollary} \label{cor:continuoussolution} Let  $\mu$ be a positive Borel measure on $\Omega$ with $\mu(\Omega) = + \infty$. Assume that the following two conditions are satisfied

$(i)$  $\mu$  is $\Gamma$-diffuse  with respect to the $m$-Hessian capacity with $\Gamma$  satisfying the  Dini type condition (\ref{eq:DiniConditionMu}), 

$(ii)$  the Dirichlet problem (\ref{eq:DirPb}) admits a bounded subsolution $\psi \in \mathcal{SH}_m(\Omega) \cap L^{\infty} (\Omega)$ i.e.  $(dd^c \psi)^m \wedge \beta^{n-m} \geq \mu$ weakly in $\Omega$ and $\psi = 0$ on $\partial \Omega$.

 Then   for any    continuous boundary datum $g \in \mathcal C^0 (\partial \Omega)$, the Dirichlet problem (\ref{eq:DirPb}) admits a unique solution  $U = U_{\mu,g} \in \mathcal{SH}_m (\Omega) \cap \mathcal C^0  (\bar{\Omega})$.
\end{corollary}
\begin{proof} Let $(K_j)_{j \in \N}$ be  an increasing sequence of relatively compact Borel subsets of  $\Omega$ such that $\Omega = \cup_j K_j$.
Set $\mu_j := {\bf 1}_{K_j} \mu$ for $j \in \N$.  By Theorem  1,  for each $j \in \N$ there exists  $u_j \in \mathcal{SH}_m(\Omega) \cap \mathcal{C}^0 (\bar{\Omega})$ such that $(dd^c u_j)^m \wedge  \beta^{n-m} = \mu_j$ and $u_j = g$ on $\partial \Omega$. By the comparison principle, $(u_j)$ is a decreasing sequence.

On the other hand let $w_g$ the maximal $m$-subharmonic function in $\Omega$ with $w_g = g$ in $\partial \Omega$. Then $\psi_g := \psi + w_g \in \mathcal{SH}_m(\Omega) \cap L^{\infty} ({\Omega})$ and $(dd^c \psi_g)^m \wedge \beta^{n-m} \geq \mu \geq \mu_j$. By the comparison principle it follows that $w_g \geq u_j \geq \psi_g$ on  $\Omega$.
Therefore $(u_j)$ decreases to a function $u \in \mathcal{SH}_m(\Omega) \cap L^{\infty} ({\Omega})$ such that $u=g$ in $\partial \Omega$  and $(d^c u)^m \wedge \beta^{n-m} = \mu$ weakly on $\Omega$. 

We need to prove that $u$ is  continuous in $\bar{\Omega}$. Indeed  choose $K_j := \{ \psi < -\varepsilon_ j\}$, for $j \in \N$,  where $(\varepsilon_j)_{j \in \N}$ is decreasing sequence of positive numbers converging to $0$ so that $\mu (\{ \psi = -\varepsilon_ j\}) = 0$  for any $j \in \N$. Then $ \psi_j := u_j + \max \{\psi,-\varepsilon_ j\}  \in \mathcal{SH}_m(\Omega) \cap L^{\infty} ({\Omega})$, $\psi_j = g$ inn $\partial \Omega$ and we have
$$
(dd^c \psi_j)^m \wedge \beta^{n-m} \geq (dd^c u_j)^m \wedge \beta^{n-m} + (dd^c  \max \{\psi,-\varepsilon_ j\})^m \wedge \beta^{n-m} \geq \mu,
$$
 weakly on $\Omega$.
 
By the comparison principle it follows that 
$$u_j + \max \{\psi,-\varepsilon_ j\} \leq u \leq u_j \, \, \, \text{on} \, \, \,  \bar{\Omega}.
$$ 
This proves that $(u_j)$ converges   to $u$ uniformly on $\bar{\Omega}$, hence $u$ is continuous in $\bar{\Omega}$.
\end{proof}
 
\subsection{Weak uniform stability theorem}
The role of the weak stability theorem $L^{1}$-$L^{\infty}$ was discovered  in \cite{EGZ09}, were  the H\"older continuity of the solution to the Dirichlet problem for the complex Monge-Amp\`ere equation on compact homogeneous manifolds was proved. Since then, this result became the main tool in deriving estimates on the modulus of continuity of solutions to  the complex Monge-Amp\`ere and Hessian equations.

In order to estimate the modulus of continuity of the solution in this general context, we need to prove a similar result.

Denote by
\begin{equation} \label{eq:integral-mu}
J_\Gamma  (\tau) := \int_0^\tau \, \frac{\gamma^{ 1 \slash m} (t)}{t } d t, \, \, \tau \in \R^+.
\end{equation}

\begin{theorem} \label{thm:stability} Let  $\mu$ be a positive  Borel measure on $\Omega$ with finite mass. Assume that $\mu$ is $\Gamma$-diffuse  with respect to the $m$-Hessian capacity and  satisfies the  Dini type condition (\ref{eq:DiniConditionMu}).

Let   $u, v\in \mathcal{SH}_m (\Omega) \cap L^{\infty} (\Omega)$ be  such that $ \liminf_{z\in\partial\Omega}(u-v)(z)\geq 0 $ and
$$
 (dd^c u)^m \wedge \beta^{n-m} \leq \mu,
$$
 in the sense of currents on $\Omega$. 
 
 Then
 $$
 \sup_{\Omega} (v-u)_+ \leq B h_\Gamma  (e^m \Vert v - u)_+\Vert_{m,\mu}^m)
 $$
 where $B > 0$ is a uniform constant, $\Vert (v - u)_+\Vert_{m,\mu}^m := \int_\Omega  (v - u)_+^m d \mu$ and $h_\Gamma$ is the reciprocal  of the  function $s \longmapsto s^{2m}  J_\Gamma^{-1} (s)$ on $\R^+$.
\end{theorem}
Observe that $h$ is continuous increasing  in $\R^+$ and $h(0) = 0$. 
\begin{proof}  We fix $\varepsilon > 0$ and apply Lemma \ref{lem:CapacityStability} to obtain the estimate
$$
\sup_{\Omega}(v-u)_+ \leq\varepsilon+ B \int_0^{ \varsigma (\varepsilon)} \frac{\gamma^{1 \slash m}  (t)}{t} d t,
$$
where
$$
\varsigma (\varepsilon) := e^m   \, \mathrm{c} _{m} (\{v - u>\varepsilon\}, \Omega).
$$
To estimate $\varsigma (\varepsilon)$ we apply  Lemma \ref{lem:CapEst} with $s = t = \varepsilon \slash 2$ which yields
\begin{eqnarray}
\mathrm{c} _{m} (\{v - u>\varepsilon\}, \Omega)&\leq& 2^m \varepsilon^{-m} \int_{v - u \geq \varepsilon\slash 2} (dd^c u)^m \wedge \beta^{n-m} \nonumber \\
&\leq & 2^{2 m} \varepsilon^{-2 m} \int_\Omega (v - u)_+^m (dd^c u)^m \wedge \beta^{n-m}.
\end{eqnarray}
Hence
$$
\varsigma (\varepsilon) \leq 2^{2 m}  e^m \varepsilon^{-2 m } \Vert (v - u)_+\Vert_{m,\mu}^m.
$$
Then by Lemma \ref{lem:CapacityStability}  
$$
\sup_{\Omega}(v-u)_+ \leq \varepsilon+ B J_\Gamma \left( 2^m e^m \varepsilon^{-2 m} \Vert (v - u)_+\Vert_{m,\mu}^m\right).
$$
Therefore if we choose $\varepsilon := h_\Gamma(2^m e^m  \Vert (v - u)_+\Vert_{m,\mu}^m)$ we obtain the required estimate.
\end{proof}

\begin{corollary} Let  $\mu$ be a positive  Borel measure on $\Omega$ with finite mass. Assume that $\mu$ is $\Gamma$-diffuse   with respect to  the $m$-Hessian capacity with $\Gamma (t) := t^{1 + a}$, where $a > 0$.
Let   $u, v\in \mathcal{SH}_m (\Omega) \cap L^{\infty} (\Omega)$ be  such that $ \liminf_{z\in\partial\Omega}(u-v)(z)\geq 0 $ and
$$
 (dd^c u)^m \wedge \beta^{n-m} \leq \mu,
$$
 in the sense of currents on $\Omega$. 
 
 Then there exists a uniform constant $A  = A (a,m) >0$ such that
 $$
 \sup_{\Omega} (v-u)_+ \leq A \left(\Vert (v - u)_+\Vert_{m,\mu}\right)^{\nu}
 $$
 where $\Vert (v - u)_+\Vert_{m,\mu} := \left[\int_\Omega  (v - u)_+^m d \mu\right]^{1 \slash m}$ and $\nu := \frac{a }{2 a  +1}$.
\end{corollary}

\begin{proof} It is a straightforward consequence of Theorem \ref{thm:stability}. Indeed it is enough to compute $h_\Gamma$ in this case. A simple computation shows that $J_\Gamma (\tau) = \frac{m}{a} \tau^{a\slash m}$. Hence 
$h_\Gamma (t) = C(a,m)  t^{\frac{a}{2 a m +m}}$ for $t > 0$.
\end{proof}

\section{ Mass estimates for Hessian measures}

For the proof of Theorem 2, we will use the same method  as in \cite{BZ20} which was inspired by an idea in \cite{KN20b}. However, since our measure does not have a compact support nor a bounded mass, we need to use the control on the behaviour of the mass of the $m$-Hessian  measure of the subsolution close to the boundary, given by Lemma \ref{lem:ComparisonIneq}.  
 
 \subsection{Proof of Theorem 2}
Recall the volume estimate stated before.

Let us fix $0 < r < m \slash (n-m)$ and $0 < b < 2 n$ and define the following function on $\R^+$:
\begin{equation}\label{eq:estimatefunction}
\ell_m (t) := \left\{\begin{array}{lcl} 
 t^{r}, \,  \, \, \, \, \, \, \, \, \, \hbox{if} \, \,   1 \leq m < n, \\
   \exp ( -b \, t^{- 1 \slash n}), \, \, \,  \hbox{if} \, \,  m = n.
\end{array}\right.
\end{equation}
 
Then the estimates  (\ref{eq:DK}) and (\ref{eq:ACKPZ}) can be written as follows: there exists a constant $B_m > 0$  such that for any Borel set $S \subset \Omega$, 
 \begin{equation} \label{eq:volumeestimate}
 \lambda_{2 n} (S) \leq B_m  \,  \ell_m \left(\text{c}_m (S,\Omega)\right) \, \text{c}_m (S,\Omega),
  \end{equation}
  where $B_m$ depends on $m,r$ and $\Omega$ when $m<n$ and  $B_n$ depends on $n, b$ and $\Omega$.

 Recall that $\varphi \in \mathcal{SH}_m (\Omega)  \cap \mathcal{C}^0 (\bar{\Omega})$ with $\varphi = 0$ on $\partial \Omega$ and we want to estimate the mass of the Hessian measure $\sigma_m (\varphi)$ on compact sets in $\Omega$.
 \begin{proof} 
We extend $\varphi$ as a continuous function in the whole of $\C^n$ with the same modulus of continuity and denote by $\varphi$ the extension.
 Then denote by $\varphi_{\delta}$ ($0 < \delta < \delta_0$) the smooth approximants of $\varphi $ in $\C^n$, defined  for $z \in \bar \Omega$ by the formula
   $$
    \varphi_{\delta}(z)=  \int_{\C^n}\varphi(\xi)\chi_{\delta}(z-\xi) d\lambda(\xi).
    $$

 Observe that for $0 < \delta < \delta_0$ and $z\in\Omega_{\delta}  := \{z \in \Omega ; \mathrm{dist} (z, \partial \Omega) > \delta\}$, 
     $$
     \varphi_{\delta}(z)=  \int_{\Omega}\varphi(z - \zeta)\chi_{\delta}(\zeta) d\lambda(\zeta),
     $$
  and then $\varphi_\delta \in  \mathcal{SH}_m (\Omega_{\delta})\cap\mathcal{C}^{\infty}(\C^n)$.
 
 Since $\varphi \in C^{0} (\bar\Omega)$, we have $\varphi_\delta \leq \varphi + \kappa_{\varphi} (\delta)$ on $\Omega$.

 We consider the $m$-subharmonic envelope of   $\varphi_\delta$ on  $\Omega$  defined by the formula 
 $$
 \psi_\delta := \sup \{\psi \in \mathcal{SH}_m (\Omega) ; \psi \leq \varphi_\delta \, \, \, \hbox{on} \, \, \,  \Omega \}\cdot 
 $$ 
 
 It follows from  \cite[Theorem 3.3]{BZ20}  that $\psi_\delta \in \mathcal{SH}_m (\Omega)$ and $\psi_\delta \leq \varphi_\delta$ on $\Omega$. 
  
  Fix   $0 < \delta < \delta_0$  and a compact set $K \subset \Omega_\delta$ and consider the following set
 
 $$
 E :=\{3 \, \kappa (\delta)  u_K^*+ \psi_\delta<\varphi- 2 \, \kappa (\delta)\} \subset \Omega.
 $$

 Since $ \kappa$ is the modulus of continuity of $\varphi $ on $\bar \Omega$, we have  $\varphi -  \kappa (\delta) \leq \varphi_{\delta}  \leq \varphi + \kappa (\delta)$ in $\Omega$ and then   $\varphi -  \kappa (\delta) \leq \psi_\delta \leq  \varphi_{\delta}  \leq \varphi  + \kappa (\delta)$  in $\Omega$. 
Therefore  $\liminf_{z \to \partial \Omega} (\psi_\delta - \varphi + \kappa (\delta)) \geq 0$, and then $E \Subset \Omega$. By the comparison principle, we conclude that
 
\begin{eqnarray} \label{eq:fundmentalestimate}
 \int_{E}(dd^c\varphi)^m\wedge\beta^{n-m} & \leq & \int_{E}(dd^c(3 \kappa (\delta)  u_K^* + \psi_{\delta}))^m\wedge\beta^{n-m} \nonumber \\
 & \leq & 3 \kappa (\delta) L \int_{E}(dd^c(u_K^* + \psi_{\delta}))^m\wedge\beta^{n-m} \\
 &+&\int_{E}(dd^c\psi_{\delta})^m\wedge\beta^{n-m}, \nonumber
 \end{eqnarray}
 where $L := \max_{0 \leq j \leq m - 1} (3 \kappa (\delta_0))^j$.
 
 Observe that $-1 + \varphi -  \kappa (\delta) \leq u_K^* + \psi_\delta  \leq \varphi + \kappa (\delta)$ on $\Omega$, hence   $\vert u_K^* + \psi_\delta\vert \leq \sup_{\Omega}  \vert \varphi\vert + 1 + \kappa (\delta_0) =: M_0$ on $\Omega$. 
 
 Therefore from inequality (\ref{eq:fundmentalestimate}), it follows that
 
\begin{equation} \label{eq:finalestimate1}
  \int_{E}(dd^c\varphi)^m\wedge\beta^{n-m} \leq 3 \kappa (\delta) L  M_0^m  \text{c}_m (E,\Omega) + \int_{E}(dd^c\psi_{\delta})^m\wedge\beta^{n-m}.
\end{equation}
 
 Moreover we have
 
 \begin{equation}\label{eq:1}
 dd^c\varphi_{\delta}\leq\frac{M_1 \kappa (\delta)}{\delta^{2}}  \beta, \, \, \, \mathrm{pointwise \, on} \, \, \, \Omega,
 \end{equation}
 where $M_1 > 0$ is a uniform constant depending only on $\Omega$.

By \cite[Theorem 3.3]{BZ20}, we have

\begin{equation}\label{eq:2}
 (dd^c\psi_{\delta})^m\wedge\beta^{n-m}  \leq  (\sigma_m (\varphi_{\delta}))_+ \leq\frac{M_1^m \kappa (\delta)^{m }}{\delta^{2 m}}  \beta^n,
 \end{equation}
 in the sense of currents on $ \Omega$.
 
 Therefore  
\begin{equation}\label{eq:3}
\int_{E}(dd^c\psi_{\delta})^m\wedge\beta^{n-m} \leq   M_1^{m} \kappa (\delta)^{m } \delta^{-2 m} \lambda_{2 n} (E).
 \end{equation}

Let us denote for simplicity $c_m (\cdot) :=  \text{c}_m (\cdot,\Omega)$. Then from (\ref{eq:finalestimate1}) and (\ref{eq:1}), we deduce that
\begin{equation}\label{eq:finalestimate2}
\int_{E}(dd^c\varphi)^m\wedge\beta^{n-m} \leq  3 \kappa (\delta) L M_0^m  \text{c}_m (E)+ M_1^{m} \kappa(\delta)^m \delta^{-2 m} \lambda_{2 n} (E) 
\end{equation} 

From (\ref{eq:finalestimate2}) and (\ref{eq:volumeestimate}), it follows that 
\begin{eqnarray} \label{eq:finalestimate3}
\int_{E}(dd^c\varphi)^m\wedge\beta^{n-m} & \leq &   3 \kappa (\delta) L M_0^m  \text{c}_m (E) \nonumber\\
&+&  B_m M_1^{m} \kappa(\delta)^m \delta^{-2 m} \ell_m (\text{c}_m (E)) \text{c}_m (E). 
\end{eqnarray}

Since  $\varphi - \kappa (\delta) \leq \psi_\delta  \leq  \varphi_{\delta}  \leq \varphi + \kappa (\delta)$  on $\Omega$, it follows that  
 $ E \subset\{u_K^* < - 1 \slash 3\}.$ 
 
The comparison principle yields the following estimate :

 \begin{equation}\label{eq:4}
  \text{c}_m  (E,\Omega) \leq 3^{m}   \text{c}_m (K,\Omega).
 \end{equation}
Indeed fix $v \in SH_m (\Omega)$ with $- 1 \leq v \leq 0$. Then $ E \subset\{ 3 u_K^* < - 1 \} \subset \{3 u_K^* < v\} \Subset \Omega$ and  the comparison principle implies that
 $$
 \int_{E} (dd^c v)^m \wedge \beta^{n - m} \leq  \int_{\{ 3 u_K^* <  v\}} 3^m  (dd^c u_K^*)^m \wedge \beta^{n - m} \leq 3^m \text{c}_m (K,\Omega).
 $$
Taking the supremum over $v$ we obtain the estimate (\ref{eq:4}).

Since $K \setminus \{u_K < u_K^*\} \subset  \{ u_K^* = - 1\} \subset E$ and $K \cap \{u_K < u_K^*\}$ has zero capacity, we see that $\int_K (dd^c\varphi)^m\wedge\beta^{n-m} \leq  \int_E (dd^c\varphi)^m\wedge\beta^{n-m}$. 

Therefore  we finally deduce from (\ref{eq:finalestimate2}), (\ref{eq:volumeestimate}) and  (\ref{eq:2}) that  for a fixed $0 < \delta < \delta_0$ and any compact set $K \subset \Omega_\delta$,  we have 
\begin{eqnarray*}
\int_K (dd^c\varphi)^m\wedge\beta^{n-m} &\leq & A_0 \kappa (\delta)  c_m (K) \\
&+& A_1  \kappa (\delta)^m \delta^{-2m}   \ell_m (3^m c_m (K)) \,  c_m (K),
\end{eqnarray*} 
where $A_0 :=  3^{m + 1}  L M_0^m$ and $A_1 :=  B_m M_1^{m}  3^{m }  $. 

By inner regularity of the capacity, we deduce that the previous estimate holds for any Borel subset $S \subset \Omega_\delta$ i.e.
\begin{eqnarray} \label{eq:estimate1}
\int_S (dd^c\varphi)^m\wedge\beta^{n-m} &\leq & A_0 \kappa (\delta)  c_m (S) \\
& +&  A_1 \kappa (\delta)^m \delta^{-2m}   \ell_m (3^m c_m (S)) \,  c_m (S). \nonumber
\end{eqnarray} 

Let $K \subset \Omega$ be any fixed compact set and $0 < \delta < \delta_0$. Then

$$
\int_K (dd^c\varphi)^m\wedge\beta^{n-m}  = \int_{K \cap \Omega_\delta} (dd^c\varphi)^m\wedge\beta^{n-m}  + \int_{K \setminus \Omega_\delta} (dd^c\varphi)^m\wedge\beta^{n-m}.
$$

We will estimate each term separately. By (\ref{eq:estimate1}) the first term is estimated easily: for $0 < \delta < \delta_0$, we have

\begin{equation} \label{eq:estimate2}
\int_{K \cap \Omega_\delta} (dd^c\varphi)^m\wedge\beta^{n-m}  \leq A_0 \kappa (\delta) c_m (K) +  A_1  \kappa (\delta)^m \delta^{-2m}   \ell_m (3^m c_m (K)) \,  c_m (K).
\end{equation} 

To estimate the second term we apply Lemma \ref{lem:ComparisonIneq} for the Borel set $ S := K \setminus \Omega_\delta$. Since $\delta_S (\partial \Omega) \leq \delta$ we get
$$
\int_{K \setminus \Omega_\delta} (dd^c\varphi)^m\wedge\beta^{n-m}  \leq \kappa (\delta)^{m} c_m (K).
$$


Therefore we finally deduce from (\ref{eq:finalestimate2}), (\ref{eq:finalestimate3}), (\ref{eq:DK}) and  (\ref{eq:2}) that  for a fixed $0 < \delta < \delta_0$ and any compact set $K \subset \Omega$,  we have 
\begin{eqnarray}\label{eq:finalestimate4}
\int_K (dd^c\varphi)^m\wedge\beta^{n-m} & \leq & A_0' \, \kappa(\delta) \,   c_m (K)\\
&  + &  A_1 \,  \kappa (\delta)^{m}  \, \delta^{-2m} \,  \ell_m (3^m c_m (K)) c_m(K), \nonumber
\end{eqnarray} 
where $\kappa = \kappa_\varphi$,  $A_0' :=  3^{m + 1}  L M_0^m + \kappa (\delta_0)^{m-1}$.

We want to optimize the right hand side of (\ref{eq:finalestimate4}) by choosing  $\delta $ so that $\kappa(\delta)  = \kappa (\delta)^{m} \delta^{-2m} \ell_m (c_m (K)) $ i.e. 
$$
\kappa (\delta)^{1-m} \delta^{2 m}  = \ell_m (3^m c_m (K)).
$$

Observe that the function $x \longmapsto \kappa (x)^{1-m}  x^{2 m}$ is continuous on $\R^+$ and takes the values $0$ at $t=0$ and $+ \infty$ at $ t = +\infty$. Therefore for any $y > 0$, the equation $y = \kappa (x)^{1-m}  x^{2 m}$ has at least one solution $x >0$ . Let us define  
the lower inverse function of the function $x \longmapsto \kappa (x)^{1-m}  x^{2 m}$ by  the following formula:
 \begin{equation} \label{eq:lowerinverse}
 \theta_m (y) :=  \inf \{x > 0 ;  \kappa (x)^{1-m}  x^{2 m} = y \}, y > 0.
 \end{equation}
 
Choose $\delta >0$ to be  $\delta = \theta_m\left[\ell_m (3^m c_m(K)\right]$. 

Set $\vartheta_m (t) := \kappa \circ \theta_m (\ell_m (3^m t))$ and observe that if $ \delta_K(\partial \Omega) \leq \vartheta_m  (c_m(K))$, then by Lemma \ref{lem:ComparisonIneq} we get
 \begin{eqnarray} \label{eq:finalIneq1}
   \int_{K}(dd^c\varphi)^m\wedge\beta^{n-m}  \leq  [\vartheta (c_m(K))]^m \, c_m (K).
  \end{eqnarray}
 
  Now assume that $ \vartheta_m  (c_m(K)) < \delta_K (\partial \Omega) \leq \delta_0$. Then we can take $\delta := \vartheta_m  (c_m(K))$ in the  inequality (\ref{eq:finalestimate4}) and get
  
  \begin{equation}\label{eq:finalIneq2}
   \int_{K}(dd^c\varphi)^m\wedge\beta^{n-m} \leq B \, \vartheta (c_m (K) \, c_m (K)  + [\vartheta (c_m (K))]^m \,  c_m (K).
  \end{equation}
  Combining inequalities  (\ref{eq:finalIneq1} ) and (\ref{eq:finalIneq2}), we obtain the estimate of the theorem with the constant $B$ given by the following formula:
  \begin{equation} \label{eq:finalConst}
   B := A_0'  + A_1 + 1.
\end{equation}   
\end{proof}

\subsection{Some consequences}

For $1 \leq m < n$,  we recover the result of (\cite{BZ20}).
\begin{corollary}  Let $\Omega \Subset \C^n$ be a $m$-hyperconvex  domain and $\varphi\in \mathcal{SH}_m (\Omega)\cap \mathcal{C}^{\alpha} (\overline\Omega)$ such that $\varphi = 0$ in $\partial \Omega$.  

 Then for any  $0 < \epsilon <\alpha m \slash  [ (n-m) (2m + \alpha (1-m))]$,  there exists a constant $ A = A (m,n,\alpha, \epsilon, \Omega)>0$ such that for every compact $K\subset\Omega$, we have
$$
\int_ K (dd^c\varphi)^m \wedge \beta^{n-m} \leq A \left[\text{c}_m (K)\right] ^{1 + \epsilon}.
$$ 

\end{corollary}

\begin{proof}
Since $\kappa_\varphi (t) = \kappa_0 t^{\alpha}$, by Theorem 2, we have for any compact $K \subset \Omega$,
$$
\int_{K}(dd^c\varphi)^m\wedge\beta^{n-m}  \leq  A  \, \left\{\vartheta_m  (c_m (K)) + \left[\vartheta_m  (c_m (K))\right]^m\right\} \, c_m (K),
$$
where
$\vartheta_m (t) := \kappa_\varphi \circ \theta_m \circ \ell_m (3^m t)$ and $\theta_m^{-1} (t) := t^{2 m+ \alpha(1 - m)}$. 
On the other hand $\ell_m (t) = t^{r}$ with $0 < r < m \slash (n-m)$, hence $\vartheta_m (t) = (3^m t)^{\alpha r \slash [2m + \alpha (1-m)]}$.
\end{proof}
For $m=n$ we obtain a much more precise result.

\begin{corollary}  \label{cor:MongeAmpereMass} Let $\Omega \Subset \C^n$ be a hyperconvex  domain.
 Let $\varphi\in \mathcal{PSH} (\Omega)\cap \mathcal{C}^{\alpha} (\overline\Omega)$ such that $\varphi = 0$ in $\partial \Omega$.  Then for any $0 < q <  \frac{ 2 n \alpha}{  3 (2 n+  (1 - n) \alpha)}$, there exists a constant $ Q = Q (n,q,\alpha,\Omega)>0$ such that for every compact $K\subset\Omega$, 
$$
\int_ K (dd^c\varphi)^n \leq Q \, \text{c}_n (K,\Omega) \, \exp \left( - q \left[\text{c}_n (K,\Omega)\right]^{- 1 \slash n}\right).
$$ 
\end{corollary}
\begin{proof} Since $\kappa_\varphi (t) = \kappa_0 t^{\alpha}$, by Theorem 2, we have for any compact $K \subset \Omega$,
$$
\int_{K}(dd^c\varphi)^m\wedge\beta^{n-m}  \leq  A  \, \left\{\vartheta_m  (c_m (K)) + \left[\vartheta_m  (c_m (K))\right]^m\right\} \, c_m (K),
$$
where
$\vartheta_m (t) := \kappa \circ \theta_m \circ \ell_m (3^m t)$ and $\theta_m^{-1} (t) := t^{2 m+ \alpha(1 - m)}$. 

When $m = n$ we have $\ell_n (t) = \exp (-b t^{-1 \slash n})$ with $b < 2 n$, hence $\vartheta_n (t) = \exp (- q  t^{-1 \slash n})$, where $q = \frac{\alpha b }{ 3 (2 n+  (1 - n) \alpha)}$.
\end{proof}

From this result, we deduce a global exponential  integrability theorem for plurisubharmonic functions in the Cegrell class with respect to Borel measures with  H\"older continuous Monge-Amp\`ere potentials. 

Let $\mathcal{F} (\Omega)$ defined as the class of negative plurisubharmonic functions $\psi$ on $\Omega$ such that there exists a decreasing sequence  of plurisubharmonic test functions  $(\psi_j)$ in $\mathcal{E}^0 (\Omega)$ such that $\sup_j \int_{\Omega} (dd^c \psi_j)^n < + \infty$ (see \cite{Ceg04}).

 We denote by  $\dot{\mathcal{F}} (\Omega)$ the set of functions $\psi \in \mathcal{F} (\Omega)$ normalized by the condition $\int_\Omega (dd^c \psi)^n \leq 1$.
\begin{corollary}  Let $\Omega \Subset \C^n$ be a hyperconvex  domain and $\varphi\in \mathcal{PSH} (\Omega)\cap \mathcal{C}^{\alpha} (\overline\Omega)$ such that $\varphi = 0$ in $\partial \Omega$.   Let  $k > n$ and  $0 < q < q_n (\alpha) := \frac{ 2 n \alpha}{   3 (2n + (1- n) \alpha )}$. Then there exists a constant $ \tilde Q = \tilde Q (k,n,q)>0$ such that for  any  $\psi \in \dot{\mathcal{F}} (\Omega)$, 
 $$
 \int_\Omega (-\psi)^k e^{-q \psi}  (dd^c \varphi)^n \leq  \tilde  Q.
 $$
 
 In particular, for any compact subset $E \Subset \Omega$, there exists a constant $R = R (E,k,n,q) > 0$ such that for any any  $\psi \in \dot{\mathcal{F}} (\Omega)$,
 $$
  \int_E e^{-q \psi}  (dd^c \varphi)^n \leq  R.
 $$
\end{corollary} 
\begin{proof}  Indeed fix $ k > n$, and $0 < q < q_n (\alpha)$. Then we have
$$
\int_\Omega  (-\psi)^k e^{-q \psi}  (dd^c \varphi)^n = 
 \int_0^{+\infty} (k t^{k - 1} + q t^{k}) \,  e^{q t}  d t \int_{ \{\psi < - t\}} (dd^c \varphi)^n.
$$

On the other hand, by \cite{CKZ05} there exists a constant $p_n > 0$ such that for any $\psi \in \dot{\mathcal{F}} (\Omega)$ and any $t > 0$, we have 
$$\text{c}_n  ( \{\psi < - t\}) \leq p_n t^{- n} \int_\Omega (dd^c \psi)^n \leq p_n  t^{- n}.
$$ 
Choosing  $q'$ such that $q < q' < q_n(\alpha)$ and applying Corollary \ref{cor:MongeAmpereMass} with the exponent $q'$ instead of $q$, we obtain
 \begin{eqnarray*}
\int_ \Omega (-\psi)^k  e^{-q \psi}  (dd^c \varphi)^n &\leq  & Q \int_0^{+\infty} (k t^{k - n - 1} + \varepsilon t^{k - n}) \,  e^{(q - q') t}  d t \\
&=: & \tilde{Q} (k,n,q) < + \infty, 
 \end{eqnarray*}
 since $k > n$ and $q' > q$.
 
 To prove the second statement, observe first that  if $E \Subset \Omega$ is a compact subset, for any $\psi \in  \dot{\mathcal{F}} (\Omega) $,
 $$
 \int_ \Omega (-\psi)^k  e^{-q \psi}  (dd^c \varphi)^n \geq (- \max_E \psi)^k \int_ E   e^{-q \psi}  (dd^c \varphi)^n.
 $$
 To conclude  we need the following well known facts about plurisubharmonic functions: 
 \begin{itemize}
 \item  the map $\psi \longmapsto \max_E \psi$ is continuous on $\mathcal{PSH} (\Omega)$ for the $L^1_{loc} (\Omega)$-topology by Hartogs lemma;
  \item the set $\dot{\mathcal{F}} (\Omega) \subset \mathcal{PSH} (\Omega)$ is compact for the $L^1_{loc} (\Omega)$-topology  and  $\max_E \psi < 0$ for any $ \psi \in \dot{\mathcal{F}} (\Omega)$  (see \cite{Ze09}).
 \end{itemize}
 Therefore there exists a constant $m(E,\Omega) > 0$ such that $(- \max_E \psi)^k \geq m(E,\Omega)^k$ for any $\psi \in \dot{\mathcal{F}} (\Omega) $.
\end{proof}
Local exponential integrability  of plurisubharmonic functions with respect to Borel measures with  H\"older continuous Monge-Amp\`ere potentials were first obtained in \cite{DNS10}.
\section{ Modulus of continuity of the solution}

Now we are ready to prove Theorem 3 and Theorem 4 from the introduction using Theorem 2, Corollary 3.10  and Theorem \ref{thm:ModC}.

\subsection{Proofs of Theorem 3 and Theorem 4}
Recall that we are given  a Borel measure $\mu$ on $\Omega$ such that there exists $\varphi\in \mathcal{SH}_m(\Omega)\cap\mathcal{C}^{0}(\overline\Omega)$ with $\varphi_{\mid \partial \Omega} \equiv 0$ and satisfying the inequality  $ (dd^c\varphi)^m\wedge\beta^{n-m} \geq \mu$ weakly on $\Omega$. 

The goal is first to prove that if modulus of continuity  $\kappa_\varphi$ of $\varphi$ satisfies the  Dini type condition (\ref{eq:DC1}) for $1 \leq m < n$  or (\ref{eq:DC2}) for $m=n$  respectively, then  for any boundary value datum $g \in \mathcal C^0 (\partial \Omega)$, the Dirichlet problem (\ref{eq:DirPb}) has a unique continuous solution.
Moreover  when $\mu (\Omega) < + \infty$, we will give an estimate on the modulus of continuity of the solution.

Recall that the function $ h_m $ is defined by its  reciprocal as follows $\tau = h_m (t)$ is the unique solution to the following equation:
\begin{equation} \label {eq:hm}
 h_m^{-1} (\tau) := \tau^{2m} J^{-1}_m (\tau), \, \, \, \, \, \, J_m (\tau) :=  \int_0^\tau \left[\kappa_\varphi \circ \theta_m \circ \ell_m (3^m t)\right]^{1 \slash m} \frac{d t}{t},
 \end{equation}
where $\ell_m (s)$ is defined by (\ref{eq:estimatefunction}) and $\theta_m$ is the inverse of the function $t \longmapsto t^{2m} \kappa_\varphi(t)^{1 - m} $.
  
We are going to prove the two theorems at the same time since the proofs only differ in the last step.

\begin{proof}  There are two steps in the proof.

1. {\it Existence of a continuous solution.} 
 Since $\mu\leq (dd^c\varphi)^m\wedge\beta^{n-m}$ and $\varphi = 0$ on $\partial \Omega$,   it follows from Theorem 2,  that $\mu$ is $\Gamma$-diffuse with $\Gamma (t) = t \gamma_m (t)$ and $\gamma_m (t) :=  \kappa_\varphi \circ \theta_m \circ \ell_m (t)$.
 
We claim that that  the conditions (\ref{eq:DC1}) and (\ref{eq:DC2}) of Theorem 3 and Theorem 4 respectively  imply that the Dini condition (\ref{eq:DiniConditionMu}) holds for $\gamma_m$ in both cases.

Indeed assume first that $1 \leq m < n$. Then $\ell_m (3^m t) =3^{r m} t^r$, where $1 < r < m \slash (n-m)$. By the change of variable $s = \ell_m(3^m t)$ we obtain
$$
\int_{0^+}^1 \gamma (t)^{1 \slash m} \frac{d t}{t} = \frac{1}{r} \int_{0^+}^1 \kappa_\varphi^{1 \slash m} (\theta_m (s)) \frac{d s}{s}.
$$
 Then the change of variables $x = \theta_m (s)$ allows to write $s = \theta_m^{-1} (x) = x^{2 m} \kappa_\varphi^{1 - m} (x)$ which implies $\frac{d s}{s} = 2 m \frac{dx}{x} + (1-m) d \kappa_\varphi (x).$
Then an easy computation shows that 
\begin{eqnarray*}
\int_{0^+}^1 \gamma (t)^{1 \slash m} \frac{d t}{t} &= & \frac{2 m }{r} \int_{0^+}^{\theta_m(1)} \kappa_\varphi^{1 \slash m} (x) \frac{d x}{x}  \\
&+ & \frac{ m(1-m)  }{r}\kappa_\varphi^{1 \slash m}(\theta_m(1)). 
\end{eqnarray*}

This proves that the condition (\ref{eq:DC1})  of Theorem 3 is equivalent to  the Dini condition (\ref{eq:DiniConditionMu}) for the function $\gamma = \gamma_m$.

Now assume that $m = n$. Then $\ell_n (3^n t) = e^{-b \slash 3 t^{1\slash n}}$. We set $s = \ell_n (3^n t)$. Then 
$\frac{d t}{t} = n \frac{d s}{s (-\log s)}$. Hence
$$
\int_{0^+} \gamma (t)^{1 \slash n} \frac{d t}{t} =  n \int_{0^+} \kappa_\varphi^{1 \slash n} (\theta_n (s)) \frac{d s}{s (-\log s)}.
$$
Now observe that $x = \theta_n (s)$ satisfies $s = \theta^{-1} (x) = x^{2 n} \kappa_\varphi^{1 - n} (x)$.

Since $\kappa_\varphi$ is increasing, it follows that  $s \geq c_1 x^{2n}$ and then $ x = \theta_n (s) \leq (s \slash c_1)^{1 \slash 2 n}$ for $s \in ]0,1]$. Therefore
$$
\int_{0^+}^1 \gamma (t)^{1 \slash n} \frac{d t}{t}  \leq  n \int_{0}^{a} \kappa_\varphi^{1 \slash n}  ((s\slash c_1)^{1 \slash 2 n})) \frac{d s}{s (-\log s)},
$$
Now the change of variable $x = (s\slash c_1)^{1 \slash 2 n}$ leads to the inequality

$$
\int_{0^+} \gamma (t)^{1 \slash n} \frac{d t}{t}   \leq  2 n^2 \int_{0^+} \kappa_\varphi^{1 \slash n}  (x) \frac{d x}{ x  (-\log (c_1 x))},
$$  
This shows that the condition (\ref{eq:DC2}) in Theorem 4 implies that the  Dini condition (\ref{eq:DiniConditionMu}) holds for the function $\gamma = \gamma_n$.
This proves our claim about $\gamma_n$. A more careful computation shows that actually the two conditions  (\ref{eq:DC2})  and  (\ref{eq:DiniConditionMu})  are equivalent, but we don't need that here.

 By Corollary \ref{cor:continuoussolution}, it follows that there is a unique function $u\in \mathcal{SH}_m(\Omega)\cap \mathcal{C}^{0}({\Omega}) $ such that 
$$ 
(dd^c u)^m\wedge\beta^{n-m}=\mu,
$$
in the weak sense on $\Omega$ and $ u=g $ on $\partial\Omega$.

\smallskip

{\it Step 2 : Estimation of the partial $\widehat{\kappa}$-modulus of continuity}. 
Assume that $\mu (\Omega) < + \infty$. We want to estimate the modulus of continuity of the solution $u$.

For  $0 < \delta < \delta_0$ and  denote as before by 
 $\widehat{u}_{\delta}(z)$ the mean value  of $u$ on the ball $B (z,\delta) \subset \Omega$.
By Lemma \ref{lem:approximation}, the global approximants  defined for $0 < \delta < \delta_0$ by  
$$
 \tilde{u}_{\delta}:= \left\{ \begin{array}{lcl}
\max\{\hat{u}_{\delta} -  \kappa (\delta),u \} &\hbox{on} & \Omega_{\delta}, \\
u  &\hbox{on} & \Omega\setminus\Omega_{\delta}
\end{array}\right.
$$
with $\kappa (\delta) :=  \kappa_\varphi (\delta) + \kappa_g (\sqrt{\delta}) + \delta$ satisfy the following properties:

\begin{itemize}
\item $\tilde{u}_{\delta}$ is  $m$-subharmonic and bounded on $\Omega$,
\item $\tilde{u}_{\delta}(z) = u (z)$ on $\Omega \setminus \Omega_\delta$
\item $0 \leq \tilde{u}_{\delta}(z) - u (z) \leq  \widehat{u}_\delta  (z) - u (z) \leq \tilde{u}_{\delta}(z) - u (z) + \kappa (\delta) $ for $z\in \Omega_{\delta}$.
\end{itemize}

Therefore we can apply Corollary \ref{cor:approximation} and get for $0<\delta<\delta_0$,
\begin{eqnarray} \label{eq3}
\int_{\Omega}(\tilde{u}_{\delta}-u)^m d\mu & = &   \int_{\Omega_\delta}(\tilde{u}_{\delta}-u)^m d\mu \nonumber \\
  &\leq& C_m \kappa_\varphi \circ \theta_m \left(D \delta\right).
\end{eqnarray}
Here $D >0$ is a uniform contant depending only on $m, n$ and uniform bounds on $g$ and 
$\varphi$.
 
  By Theorem \ref{thm:stability} it follows that 
 \begin{eqnarray} \label{eq4}
\sup_{\Omega}(\tilde{u}_{\delta}-u)
 & \leq & B \,  h_{m} (2^m e^m \Vert \tilde u_\delta - u\Vert_{m,\mu}^m), 
 \end{eqnarray}
 where $ B>0$ is a uniform constant,  and $h_m =  h_{\Gamma_m} $ is defined by the formula (\ref{eq:hm}).

 Therefore  from equation  (\ref{eq3}) and (\ref{eq4}) we deduce that
 \begin{equation*} 
\sup_{\Omega}( \tilde{u}_{\delta}-u)  \leq    B \, h_m \left[C_m \kappa_\varphi \circ \theta_m (D \delta)\right], 
 \end{equation*}
 which implies
 \begin{equation} \label{eq:Finaleq}
 \sup_{\Omega_{\delta}}( \tilde{u}_{\delta}-u)  \leq  B  \,   h_m \left[C_m \kappa_\varphi \circ \theta_m  \left( D \delta \right)\right].
  \end{equation}
 
 
Recall that $ \hat{u}_{\delta} - u \leq \tilde u_\delta - u + \kappa (\delta) $ on $\Omega_\delta$. Hence  for $0 < \delta < \delta_0$

  \begin{equation} \label{eq5}
   \sup_{\Omega_{\delta}}( \hat{u}_{\delta}-u) \leq \sup_{\Omega}(\tilde{u}_{\delta}-u)+ \kappa (\delta) .
  \end{equation}
  Using (\ref{eq:Finaleq}) and (\ref{eq5}) we finally get for $0 < \delta < \delta_0$
  \begin{equation} \label{eq:MC-estimate}
   \sup_{\Omega_{\delta}}( \hat{u}_{\delta}-u) \leq B \tilde{\kappa}_m (\delta),
 \end{equation}
 where 
\begin{equation} \label{eq:MC-Solution}
\tilde \kappa_m (\delta) :=   h_m \left[C_m \kappa_\varphi \circ \theta_m  \left(D\delta \right)\right]  +  \kappa_\varphi (\delta) + \kappa_g (\sqrt{\delta}),
 \end{equation}
 and $h_m$ is defined by (\ref{eq:hm}).
 \end{proof}

\smallskip

 When $g$ is $C^{1,1} (\partial \Omega)$, we can improve the estimates on the modulus of continuity  of the solution in Theorem 3 and Theorem 4. 
\begin{theorem} \label{thm:ModC+} Under the same assumptions as Theorem 3 (resp. Theorem 4)  and assume moreover  that $g \in C^{1,1} (\partial \Omega)$. Then the solution $u = U_{g,\mu}$ is continuous and its $\hat{\kappa}$-modulus of continuity satisfies the following estimate for  $0 < \delta < \delta_0$,
$$
   \sup_{\Omega_{\delta}}( \hat{u}_{\delta}-u) \leq B \, \, {\kappa'}_m (\delta),
 $$
 where $ B > 0$ is a uniform constant and 
 $${\kappa'}_m (\delta):=  h_m \left[C_m \kappa_\varphi \circ \theta_m  \left(D\delta^2 \right)\right]  +  \kappa_\varphi (\delta).
 $$ 
\end{theorem}
\begin{proof}  Indeed, using Corollary \ref{cor:improvement1} at the end of the previous proof, we get a better inequality i.e. for $0 < \delta < \delta_0$, we have
 $$
 \int_{\Omega}(\tilde{u}_{\delta}-u)^m d\mu   \leq  C_m \kappa_\varphi \circ \theta_m \left(D \delta^2\right).
 $$
 Then the inequality  (\ref{eq4})  becomes for $0 < \delta < \delta_0$,
 $$
   \sup_{\Omega_{\delta}}( \hat{u}_{\delta}-u) \leq B  \,   h_m \left[C_m \kappa_\varphi \circ \theta_m  \left( D \delta^2\right)\right]
 $$
 The rest of the proof is done in the same way and leads to the required estimate by observing that in this case the maximal $m$-subharmonic extension of $g$ is Lipschitz in $\bar \Omega$. Indeed $g$ extends as  a function 
 $g \in C^{1,1} (\partial \Omega$. Then as in the proof of Corollary \ref{cor:improvement1} we construct  two  $m$-subharmonic functions $v$ and $-w$ such that $v, w  \in C^{1,1} (\bar \Omega)$, $v \leq w$ in $\Omega$ and $v = g =  w$ in $\partial \Omega$. 
 Then  from the proof of Lemma \ref{lem:approximation}  we see that we can replace $\kappa_g (\sqrt{\delta})$ by $\delta$ in Corollary \ref{cor:approximation}.
\end{proof}

\subsection{Some consequences}
Let us state  corollaries of Theorem 3 and Theorem 4 to show how the  estimates obtained so far are more precise compared to previous ones (see \cite{KN20b}, \cite{BZ20}).

\begin{corollary} \label{cor:HolderHE}
Let $\Omega \Subset \C^n$ be a  bounded strictly $m$-pseudoconvex  domain  with $1 \leq  m \leq n$ and $\mu $ a positive Borel measure on $\Omega$. Assume that there exists $\varphi\in \mathcal{SH}_m(\Omega)\cap\mathcal{C}^{\alpha}(\overline\Omega)$ such that  
\begin{equation} \label{eq:subsol2}
 \mu \leq (dd^c\varphi)^m\wedge\beta^{n-m}, \, \, \, \mathrm{weakly \, \, on} \, \,  \Omega \, \, \, \mathrm{and} \, \, \, \varphi_{\mid{\partial \Omega}} \equiv 0.
\end{equation}

 Then for any continuous function $g \in \mathcal{C}^{2 \alpha} (\partial \Omega)$, there exists a unique function $U = U_{g,\mu} \in \mathcal{SH}_m (\Omega) \cap \mathcal{C}^{0} (\bar{\Omega})$ such that  
 $$
 (dd^c U)^m\wedge\beta^{n-m} = \mu, \, \, \, \mathrm{and} \, \, \, U = g \, \, \, \mathrm{on} \, \, \,  \partial \Omega.
 $$
 Moreover if $\mu (\Omega) < + \infty$, $U \in \mathcal{C}^{\tilde{\alpha}} (\bar{\Omega})$ for any $\tilde{\alpha} <  \tilde{\alpha}_m$,
where

 \begin{equation} \label{eq:Holderexp}
 \tilde \alpha_m := \frac{ \tilde r   \alpha^2}{ m \tilde{m} \left[ \tilde m   + 2 \alpha \tilde r\right]},
 \end{equation}
and  $ \tilde r:= \frac{m}{n-m}$ and $\tilde m := 2 m + \alpha(1-m)$.
 \end{corollary}

 \begin{proof}  We want to apply  Theorem 3.  Here we have $\kappa_\varphi (t) = \kappa_0 t^\alpha$, $\ell_m (t) = t^{r}$ with $0 < r < m \slash (n-m)$. By Theorem 2, $\mu$ is $\Gamma$-diffuse with $\Gamma (t) = A t  \kappa_\varphi \circ \theta_m \left(\ell_m (t)\right)$ and $\theta_m$ is the inverse  of the function $t \longmapsto t^{2m} \kappa_\varphi (t)^{1 - m} = t^{2m + \alpha (1-m)}$. Thus $\Gamma (t) = A  t^{1 + r \alpha \slash \tilde m}$.

 Then $ J_m (\tau) = A^{1\slash m} \frac{ m \tilde m}{r \alpha} \tau^{ \alpha r \slash m \tilde m} $ and 
 $h^{-1}_m (t) =A' (m,\alpha)  \, \,  t^{2 m +  m \tilde m \slash  \alpha r}$.  
 
 Finally  by (\ref{eq:MC-Solution}) the $\widehat{\kappa}$-modulus of continuity of the solution is dominated as follows 
 
$$
\widehat{\kappa}_U (\delta) \leq  C'(\alpha,m,n,\Omega) \delta^{\frac{ r \alpha^2 }{ m \tilde m \left[  \tilde m + 2 \alpha  r \right]}}. 
 $$ 
 As before  we apply Lemma \ref{lem:sup-mean} to concude.  \end{proof}
 
 \begin{corollary}  \label{cor:HolderMA} Under the assumption of Theorem 4, with $\varphi \in C^\alpha (\bar{\Omega}$ with $0 < \alpha \leq 1$ and $g \in \mathcal{C}^{2 \alpha} (\partial \Omega)$,  the solution $U := U_{g,\mu}$ to the Dirichlet problem is H\"older continuous in $\bar \Omega$ and its modulus of continuity satisfies the  following  estimate
$$
\kappa_U (\delta) \leq C \delta^{\alpha \slash  2 n \tilde n} (- \log \delta)^{1\slash 2},
$$  
where $\tilde n := ( 2-\alpha) \,  n + \alpha $ and $C > 0$ is a positive uniform constant. 
   \end{corollary}
Here $\kappa_U$ is the usual modulus of continuity of $U$ on $\bar{\Omega}$ defined as follows:
\begin{equation}
\kappa_U(\delta) := \sup \{ \vert U (z) - U(z') \vert \, ; \, z, z' \in  \bar{\Omega}, \vert z - z'\vert \leq \delta\}
\end{equation}

The precise relationship between $\kappa_U$ and $\widehat{\kappa}_U$ was discussed in section 2.3 (see \cite{Ze20} for more details).  
\begin{proof}
We want to apply  Theorem 4.  Here we have $\kappa_\varphi (t) = \kappa_0 t^\alpha$, $\ell_n (3^n t) =e^{- b' t^{-1\slash n}}$ with $0 < b' = b \slash 3 < 2 n\slash 3$. By Theorem 2, $\mu$ is $\Gamma$-diffuse with $\Gamma (t) = A_0 t  \kappa_\varphi \circ \theta_n \left(\ell_n (3^n t)\right)$ and $\theta_n$ is the inverse  of the function 
$t \longmapsto t^{2n} \kappa_\varphi (t)^{2 - n} = t^{2n + \alpha (1-n)}$.  Thus 
$$
\Gamma (t) = A_0  t  e^{-b_1 t^{-1\slash n}},
$$
where $b_1:= \alpha b' \slash [2 n + \alpha (1-n)]$.
 
 Then for $\tau > 0$, we have
 $$
 J_n (\tau) = A_0^{1\slash n}  \int_0^\tau  e^{-b_2 t^{-1\slash n}} \frac{dt}{t},
 $$
 where $b_2 := b_1\slash n$. 
 
  By the change of variable $s = t^{-1\slash n}$  we get
 $$
  J_n (\tau) = n A_0^{1\slash n}  \int_{\tau^{-1\slash n}}^{+ \infty}  e^{-b_2 s} \frac{ds}{s}.
 $$
  Fix $\tau_0 > 0$. Then for $0< \tau < \tau_0$ we have  
  
  $$
  J_n (\tau) \leq A_1  \int_{\tau^{-1\slash n}}^{+ \infty}  e^{-b_2 s} {ds} = A_2  e^{-b_2 \tau^{-1\slash n}},
  $$
  where $A_1 := n A_0^{1\slash n} \tau_0^{1\slash n}$ and $A_2 := A_1 \slash b_2$.
  
 Therefore given $\varepsilon > 0$, there exists $y_0 > 0$ and a constant $A_3 > 0$ such that  for $0 < y < y_0$ we have
  $$
h^{-1}_n (y) :=  y^{2n}  J_n^{-1} (y) \geq \frac{b_2 y^{2 n}}{\left(- \log (y\slash A_2)\right)^n}\cdot
  $$
An easy computation shows that there exists $x_0 > 0$ small enough such that for $0< x < x_0$,  $h_n (x)  \leq A_3 x^{1\slash 2 n} \, (- \log x)^{1 \slash 2}$.

By Theorem 4, the $\widehat{\kappa}$-modulus of continuity of the solution $U$ to the Dirichlet problem (\ref{eq:DirPb}) in this case satisfies $\widehat{\kappa}_U (\delta) \leq \widehat{\kappa} (\delta)$, for $0 < \delta < \delta_0$,  where
$$
\widehat{\kappa} (\delta) := A_4 \delta^{\alpha \slash 2 n  \tilde n} (- \log \delta)^{1 \slash 2}.
$$

Now we need to apply Lemma \ref{lem:sup-mean} to conclude. Indeed it's clear that the modulus of continuity $\widehat{\kappa}$ obtained above satisfies the condition  (\ref{eq:fullkappa}). Moreover the function $U$ is  $\widehat{\kappa}$-continuous near the boundary by Corollary \ref{cor:HolderHE}. \end{proof}

\begin{remark}
As before we can improve the H\"older exponent in the previous corollaries when $g \in C^{1,1} (\partial \Omega)$. Indeed, using Theorem \ref{thm:ModC+},
 we conclude that in Corollary \ref{cor:HolderHE} ,  the critical exponent is $2 \tilde \alpha_m$, while in Corollary \ref{cor:HolderMA} the modulus of continuity of the solution satisfies the following inequality :
$$
\kappa_U (\delta)  \leq C \delta^{\alpha \slash  n \tilde n} (- \log \delta)^{1\slash 2},
$$
for $0 < \delta < \delta_0$.
\end{remark}

\smallskip

\smallskip

 {\bf Acknowledgements:} This work is a natural continuation of the work done recently by  Amel Benali and the second author on the same subject. Some results and ideas from the latter  has been used here.  

It has been completed after the recent works of S. Ko\l odziej and N. Cuong Nguyen \cite{KN20a}, \cite{KN20b} which were a source of fruitful inspiration. We are grateful to them.

We are indebted to Chinh Hoang Lu for a careful checking of the first version of this paper and for useful comments. We thank Ngoc Cuong Nguyen for useful remarks which made it possible  to correct the statement of Corollary 4.2 and Corollary 4.3 in the previous version of this paper. 
We also thank Eleonora Di Nezza and Vincent Guedj  for useful discussions.                                                                                                                             

Finally we would like to thank the referee for a very careful reading and many good suggestions that helped to improve 
the presentation of the paper.

\end{document}